\documentclass{cmslatex}

\usepackage[paperwidth=7in, paperheight=10in, margin=.875in]{geometry}
\usepackage[backref,colorlinks,linkcolor=red,anchorcolor=green,citecolor=blue]{hyperref}
\usepackage{amsfonts,amssymb}
\usepackage{amsmath}
\usepackage{graphicx}
\usepackage{cite}
\usepackage{enumerate}
\renewcommand{\theequation}{\arabic{section}.\arabic{equation}}

\usepackage{bm}
\usepackage{float}
\usepackage{epstopdf}
\usepackage{algorithm,algorithmic,booktabs}
\usepackage{mathtools}
\mathtoolsset{showonlyrefs}

\DeclareMathOperator*{\argmin}{arg\,min}

\numberwithin{equation}{section}

\newcommand{\mc}[1]{\mathcal{#1}}

\newcommand{\MCF}{\mathcal{F}}

\newcommand{\mbA}{\mathbb{A}}
\newcommand{\EE}{\mathbb{E}}
\newcommand{\PP}{\mathbb{P}}
\newcommand{\RR}{\mathbb{R}}
\newcommand{\NN}{\mathbb{N}}

\newcommand{\ltwonorm}[1]{\left\lVert#1\right\rVert_2}

\newcommand{\half}{\frac{1}{2}}

\newcommand{\eps}{\epsilon}

\newcommand{\abs}[1]{\left|#1\right|}
\newcommand{\average}[1]{\left\langle#1\right\rangle}

\newcommand{\ud}{\,\mathrm{d}}

\newcommand{\alphain}{\alpha^{i,n}}
\newcommand{\oon}{\frac{1}{N}}

   \allowdisplaybreaks
   
\begin{document}
   	\title{Deep Fictitious Play for Stochastic Differential Games\thanks{Received date, and accepted date (The correct dates will be entered by the editor).}}
   	
   	
   	\author{Ruimeng Hu\thanks{Department of Statistics, Columbia University, New York, NY, 10027-4690. Current position: Department of Mathematics and Department of Statistics and Applied Probability, University of California, Santa Barbara, CA, 93106-3080, rhu@ucsb.edu. RH was partially supported by the NSF grant DMS-1953035.}}

   	\pagestyle{myheadings} \markboth{Deep fictitious play for stochastic differential games}{Ruimeng Hu} \maketitle
   	
   	\begin{abstract}
   			In this paper, we apply the idea of fictitious play to design deep neural networks (DNNs), and develop deep learning theory and algorithms for computing the Nash equilibrium of asymmetric $N$-player non-zero-sum stochastic differential games, for which we refer as \emph{deep fictitious play}, a multi-stage learning process. Specifically at each stage, we propose the strategy of letting individual player optimize her own payoff subject to the other players' previous actions, equivalent to solve $N$ decoupled stochastic control optimization problems,
   			which are approximated by DNNs. Therefore, the fictitious play strategy leads to a structure consisting of $N$ DNNs, which only communicate at the end of each stage. The resulted deep learning algorithm based on fictitious play is scalable, parallel and model-free, {\it i.e.}, using GPU parallelization, it can be applied to any $N$-player stochastic differential game with different symmetries and heterogeneities ({\it e.g.}, existence of major players). We illustrate the performance of the deep learning algorithm by comparing to the closed-form solution of the linear quadratic game. Moreover, we prove the convergence of fictitious play  under appropriate assumptions, and verify that the convergent limit forms an open-loop Nash equilibrium. We also discuss the extensions to other strategies designed upon fictitious play and closed-loop Nash equilibrium in the end.
   	\end{abstract}
   	

\section{Introduction}\label{sec:intro} 
In stochastic differential games, a Nash equilibrium refers to strategies by which no player has an incentive to deviate. Finding a Nash equilibrium is one of the core problems in noncooperative game theory, however, due to the notorious intractability of $N$-player game, the computation of the Nash equilibrium has been shown extremely time-consuming and memory demanding, especially for large $N$ \cite{DaGoPa:2009}. On the other hand, a rich literature on game theory has been developed to study consequences of strategies on interactions between a large group of rational ``agents'', {\it e.g.}, system risk caused by inter-bank borrowing and lending, price impacts imposed by agents' optimal liquidation, and market price from monopolistic competition. This makes it crucial to develop efficient theory and fast algorithms for computing the Nash equilibrium of $N$-player stochastic differential games. 

Deep neural networks with many layers have been recently shown to do a great job in artificial intelligence ({\it e.g.}, \cite{Be:09, LeBeHi:15} ). The idea behind is to use compositions of simple functions to approximate complicated ones, and there are approximation theorems showing that a wide class of functions on compact subsets can be approximated by a single hidden layer neural network ({\it e.g.}, \cite{Pinkus:99}). This brings a possibility of solving a high-dimensional system using deep neural networks, and in fact, these techniques have been successfully applied to solve stochastic control problems \cite{HaE:16, deep:2018, deep2:2018}.

In this paper, we propose to build deep neural networks by using strategies of fictitious play, and develop parallelizable deep learning algorithms for computing the Nash equilibrium of asymmetric $N$-player non-zero-sum stochastic differential games. We consider a stochastic differential game with $N$ players, and each player $i \in \mathcal{I} :=  \{1, 2, \ldots, N\}$ has a state process $X_t^i \in \RR^d$ and takes an action $\alpha_t^i$ in the control set $A \subset \RR^k$. The dynamics of the controlled state process $X_\cdot^i$ on $[0,T]$ are given by
\begin{equation}\label{def:Xt:general}
\ud X_t^i = b^i(t, \bm{X}_t, \bm{\alpha}_t) \ud t + \sigma^i(t, \bm{X}_t, \bm{\alpha}_t) \ud W_t^i + \sigma^0(t, \bm{X}_t, \bm{\alpha}_t) \ud W_t^0, \quad X_0^i = x^i, \quad i \in \mc{I},
\end{equation}
where $\bm{W} :=[W^0, W^1,\ldots, W^N]$ are $N+1$ $m$-dimensional independent Brownian motions, 
$(b^i, \sigma^i)$ are deterministic functions: $[0,T] \times \RR^{d\times N} \times A^N \hookrightarrow \RR^d \times \RR^{d\times m}$. The $N$ dynamics are coupled since all private states $\bm{X}_t = [X_t^1, \ldots, X_t^N]$ and all the controls\footnote {Although in the literature of math finance, one usually models $b^i$ and $\sigma^i$ to only depend on player $i$'s own action, but it is common in literature of economics that player $i$'s private state is also influenced by others' actions, {\it e.g.}, $\alpha_t^i$ is a priced set by companies and $X_t^i$ is the production quantity. To be general, we include this feature in our model, which yields \eqref{def:Xt:general}.}  $\bm{\alpha}_t = [\alpha_t^1, \ldots, \alpha_t^N]$ affect the drifts $b^i$ and diffusions $\sigma^i$.

Each player's control $\alpha_t^i$ lives in the space $\mathbb{A} = \mathbb{H}^2_T(A)$ of progressively measurable $A$-valued processes satisfying the integrability condition:
\begin{equation}\label{def:admissible}
\EE[\int_0^T \abs{\alpha_t^i}^2 \ud t] < \infty.
\end{equation}
Using the strategy $\bm{\alpha} \in \mbA^N$,  the cost associated to player $i$ is of the form:
\begin{equation}\label{eq:cost}
J^i(\bm{\alpha}) := \EE\left[\int_0^T f^i(t, \bm X_t, \bm\alpha_t) \ud t + g^i(\bm X_T)\right],
	\end{equation}
where the running cost $f^i: [0,T] \times \RR^{d \times N}\times A^N \hookrightarrow \RR$ and terminal cost $g^i: \RR^{d\times N} \hookrightarrow \RR$ are deterministic measurable functions.

In solving stochastic differential games, the notion of optimality of common interest is the Nash equilibrium. A set of strategies $\bm{\alpha}^\ast = (\alpha^{1,\ast}, \ldots, \alpha^{N,\ast}) \in \mbA^N$ is called a Nash equilibrium if
\begin{equation}\label{def:Nash}
\forall i \in \mc{I} \text{ and } \beta^i \in \mbA, \quad J^i(\bm{\alpha}^\ast) \leq J^i(\beta^i, \bm{\alpha}^{-i,\ast}),
\end{equation}
where $\bm{\alpha}^{-i,\ast}$ represents strategies of players other than the $i$-th one
\begin{equation*}
\bm{\alpha}^{-i,\ast} := [\alpha^{1,\ast}, \ldots, \alpha^{i-1,\ast}, \alpha^{i+1,\ast}, \ldots, \alpha^{N,\ast}] \in \mbA^{N-1}.
\end{equation*}
In fact, depending on the space where one searches for actions (the information structure available to the players), the types of equilibria include open-loop ($\bm{W}_{[0,t]}$),  closed-loop ($\bm{X}_{[0,t]}$), and closed-loop in feedback form ($\bm{X}_t$). We start with the setup \eqref{def:Nash} which corresponds to the open-loop case. Theoretically, it is more tractable, due to the indirect nature ({\it i.e.} player $i$ will not change his strategy when player $j$'s strategy changes because player $i$ can not observe or feel the change). Practically, there are applications falling into this framework, for instance, the prisoner's dilemma from game theory. This is the scenario that when two people get arrested and investigated, they are in solitary confinements and can not communicate with each other, nor observe the other's choice. In this case, it is reasonable to assume that $\alpha_t^i$ does not depend on the past decisions $\bm \alpha_{[0,t)}$ nor the players' states $\bm X_{[0,t]}$ as these information is not available under this framework. The generalization of deep learning theory for closed-loop cases will be discussed in Section~\ref{sec:close:loop}.

An alternative method of solving $N$-player stochastic differential games is via mean-field games, introduced by Lasry and Lions in \cite{LaLi1:2006,LaLi2:2006,LaLi:2007} and by Huang, Malham\'{e} and Caines in \cite{HuMaCa:06,HuCaMa:07}. The idea is to approximate the Nash equilibrium by the solution of mean field equilibrium (the formal limit of $N\rightarrow \infty$) under mild conditions \cite{CaDe:13}, which leads to an approximation error of order $N^{-1/(d+4)}$ assuming that the players are indistinguishable, {\it i.e.}, all coefficients $(b^i, \sigma^i, f^i, g^i)$ are free of $i$. We refer to the books \cite{CaDe1:17,CaDe2:17} and the references therein for further background on mean-field games. However, beyond the case of a continuum of infinitesimal agents with or without major players, the mean-field equilibrium may not be a good approximation in general. In addition, the mean-field game often exhibits multiple equilibria, some of which do not correspond to the limit of $N$-player game as $N\rightarrow \infty$, {\it e.g.}, in the optimal stopping games \cite{NuMaTa:18}.  Moreover, when the number of players is of middle size ({\it e.g.}, $N \sim 50$), the approximation error made by the mean-field equilibrium is large while direct solvers based on forward-backward stochastic differential equations (FBSDEs) or on partial differential equations (PDEs) are still computationally unaffordable. Therefore, it is demanding to develop new theory and algorithms for solving the $N$-player game. 

The idea proposed in this paper is natural and motivated by the fictitious play, a learning process in game theory firstly introduced by Brown in the static case \cite{Br:49, Br:51} and recently adapted to the  mean field case by Cardaliaguet \cite{CaHa:17,BrCa:18} and coauthors. In the fictitious play, after some arbitrary initial moves at the first stage, the players myopically choose their best responses against the empirical strategy distribution of others' action at every subsequent stage. It is hoped that such a learning process will converge and lead to a Nash equilibrium. In fact, Robinson \cite{Ro:51} showed this holds for zero-sum games, and Miyazawa \cite{Mi:61} extended it to $2\times 2$ games. However, Shapley's famous $3\times 3$ counter-example \cite{Sh:64} indicates that this is not always true. Since then, many attempts are made to identify classes of games where the global convergence holds  \cite{MiRo:91,MoSh:96,MoSh2:96,HoMoSe:98, CrAn:03, Be:05, HoSa:02}, and where the process breaks down \cite{Jo:93,MoSe:96,FoYo:98,KrSj:98}, to name a few.

Based on fictitious play, we propose a deep learning theory and algorithm for computing the open-loop Nash equilibria. Unlike closed-loop strategies of feedback form, which can be reformulated as the solution to $N$-coupled Hamilton-Jacobi-Bellman (HJB) equations by dynamic programming principle (DPP), open-loop strategies are usually identified through FBSDEs. The existence of explicit solutions to both equations highly depends on the symmetry of the problem, in particular,  for most cases where explicit solutions are available, the players are statistically identical. Therefore, an efficient and accurate numerical scheme is crucial for solving such FBSDEs. Traditional ways run into the technical difficulty of {the curse of dimensionality}, thus are not feasible when the dimensionality goes beyond 5. Observing impressive results solved by deep learning on various challenging problems \cite{Be:09,KrSuHi:12,LeBeHi:15}, we shall use deep neural networks to overcome the curse of dimensionality for moderately large $N$ and asymmetric games. We first boil down the game into $N$ stochastic control subproblems, which are conditionally independent given past play at each stage. Since we first focus on open-loop equilbria (as opposed to closed-loop ones) in each subproblem, the strategies are considered as general progressively measurable processes (as opposed to functions of $(t, \bm{X}_t)$). Therefore, without the feedback effects, one can design a deep neural network to solve stochastic control subproblems individually. The control at each time step is approximated by a feed-forward subnetwork, whose inputs are initial states $\bm{X}_0$ and noises $\bm{W}_{[0,t]}$ in lieu of the definition of open-loop equlibria. For player $i$'s control problem, $\bm{X}^{-i}$ is generated using strategies from past, {\it i.e.}, considered as fixed while player $i$ optimizes herself.

\medskip
\noindent{\bf Main contribution.} The contribution of deep fictitious play is three-fold. Firstly, our algorithm is scalable: in each round of play, the $N$ subproblems can be solved in parallel, which can be accelerated by the feature of multi-GPU. Secondly, 
we propose a deep neural network for solving general stochastic control problem where strategies are general processes instead of feed-back form. In lack of DPP, algorithms from reinforcement learning are no longer available. We approximate the optimal control directly in contrast to approximating value functions \cite{Po:07}. Thirdly, the algorithm can be applied to asymmetric games, as for each player, there is a corresponding neural network.

\medskip
\noindent{\bf Related literature.} 
Most literature in deep learning and reinforcement learning algorithms in stochastic control problems uses DPP with which, the problem can be solved backwardly, {\it i.e.}, to find the optimal control at the terminal time, and then decide the previous decision. 
Among them, let me mention the recent works \cite{deep:2018, deep2:2018}, which approximate the optimal policy by neural networks in the spirit of deep reinforcement learning, and the approximated optimal policy is obtained in a backward manner. While in our algorithm, we stack these subnetworks together to form a deep network and train them simultaneously. In fact, our structure is inspired by Han and E \cite{HaE:16}. The difference is that they feed the network with $X_t$ seeking for feedback-form controls, while we feed the initial states $X_0$ and noises $W_{[0,t]}$ for each player's network, seeking for open-loop Nash equilibrium. In terms of using fictitious play to solve multi-agent problems, \cite{HeSi:16,LaZa:17,MgJeCo:18} design reinforcement learning algorithms assuming the system \eqref{def:Xt:general} is {unknown}; while our algorithm needs the knowledges of $b^i$, $\sigma^i$, $f^i$ and $g^i$.

\medskip
\noindent{\bf Organization of the paper.} In Section~\ref{sec:DFP}, we systematically introduce the deep fictitious play theory, and implementation of deep learning algorithms using Keras with GPU acceleration. In Section~\ref{sec:LQ}, we apply deep fictitious play to linear quadratic games, and prove the convergence of fictitious play under proper assumptions on parameters, with the limit forming an open-loop Nash equilibrium. Performance of deep learning algorithms are presented in Section~\ref{sec:numerics}, where we simulate stochastic differential games with a large number of players ({\it e.g.}, $N=24$). We make conclusive remarks, and discuss the extensions to other strategies of fictitious play and closed-loop cases in Section~\ref{sec:rmk}.

\section{Deep fictitious play}\label{sec:DFP}
In this section, we describe the theory and algorithms of deep fictitious play, which by name, is known to build on fictitious play and deep learning. We first summarize all the notations that shall be used as below. Given a probability space $(\Omega, \MCF, \PP)$, we consider 
\begin{itemize}
\item $\bm{W} = [W^0, W^1, \ldots, W^N]$, a $(N+1)$-vector of $m$-dimensional independent Brownian motions;
\item $\mathbb{F} = \{\MCF_t, 0\leq t \leq T\}$, the augmented filtration generated by $\bm{W}$;
\item $\mathbb{H}^{2}_T(\RR^d)$, the space of all progressively measurable $\RR^d$-valued stochastic processes $\alpha: [0,T] \times \Omega \hookrightarrow \RR^d$ such that $\ltwonorm{\alpha} = \EE[\int_0^T \abs{\alpha_t}^2 \ud t ] < \infty$.
\item $\mbA = \mathbb{H}^{2}_T(A)$, the space of admissible strategies, {\it i.e.}, elements in $\mbA$ satisfy \eqref{def:admissible}. $\mbA^N = \mbA \times \mbA \times \ldots \times \mbA$, a product of $N$ copies of $\mbA$;
\item $\bm{\alpha} = [\alpha^1, \alpha^2, \ldots, \alpha^N]$, a collection of all players' strategy profiles. With a negative superscript, $\bm{\alpha}^{-i} = [\alpha^1, \ldots, \alpha^{i-1}, \alpha^{i+1}, \ldots, \alpha^{N}]$ means  the strategy profiles excluding player $i$'s. If a non-negative superscript $n$ appears ({\it e.g.}, $\bm{\alpha}^n$), this N-tuple stands for the strategies from stage $n$. When both exist,  $\bm{\alpha}^{-i,n} =  [\alpha^{1,n}, \ldots, \alpha^{i-1,n}, \alpha^{i+1,n}, \ldots, \alpha^{N,n}]$ is a $(N-1)$-tuple representing strategies excluding player $i$ at stage $n$. We use the same notations for other stochastic processes ({\it e.g.}, $\bm{X}^{-i}, \bm{X}^n$);

\end{itemize}

We assume that the players start with an initial smooth belief $\bm{\alpha}^0 \in \mbA^N$. At the beginning of stage $n+1$, $\bm{\alpha}^n$ is observable by all players. Player $i$ then chooses best response to her beliefs about opponents described by their play at the previous stage $\bm{\alpha}^{-i,n}$. Then, player $i$ faces an optimization problem:
\begin{equation}\label{def:J:SFP}
\inf_{\beta^i \in \mbA} J^{i}(\beta^i;\bm{\alpha}^{-i,n}), \quad 
J^{i}(\beta^i; \bm\alpha^{-i,n}) = \EE\left[\int_0^T f^i(t,\bm{X}_t^\alpha, (\beta^i_t, \bm{\alpha}_t^{-i,n})) \ud t + g^i(\bm X_T^\alpha)\right],
\end{equation}
where $\bm X_t^\alpha = [X_t^{1,\alpha}, X_t^{2,\alpha}, \ldots, X_t^{N,\alpha}]$ are state processes controlled by $(\beta^i, \bm\alpha^{-i,n})$:
\begin{align}
\ud X_t^{\ell,\alpha}= b^\ell(t, \bm{X}_t^\alpha, (\beta^i_t, \bm{\alpha}_t^{-i,n})) \ud t & +  \sigma^\ell(t, \bm{X}_t^\alpha, (\beta^i_t, \bm{\alpha}_t^{-i,n})) \ud W_t^\ell  \nonumber  \\
 &+  \sigma^0(t, \bm{X}_t^\alpha, (\beta^i_t, \bm{\alpha}_t^{-i,n})) \ud W_t^0, \; X_0^{\ell,\alpha} = x^\ell,
\end{align}
for all $\ell \in \mc{I}$. Denote by $\alpha^{i,n+1}$ the minimizer in \eqref{def:J:SFP}:
\begin{equation}\label{def:SFP}
\alpha^{i,n+1} := \argmin_{\beta^i \in \mbA} J^{i}(\beta^i; \bm{ \alpha}^{-i,n}), \quad \forall i \in \mc{I}, n \in \mathbb{N},
\end{equation}
we assume $\alpha^{i,n+1}$ exists throughout the paper. More precisely, $\alpha^{i,n+1}$ is the player $i$'s optimal strategy at the stage $n+1$ when her opponents dynamics \eqref{def:Xt:general} evolve according to $\alpha^{j,n}$, $j \neq i$. All players find their best responses simultaneously, which together form $\bm{\alpha}^{n+1}$.

\begin{rem}
	Note that the above learning process is slightly different than the usual simultaneous fictitious play, where the belief is described by the time average of past play: $\frac{1}{n}\sum_{k=1}^n \bm{\alpha}^{-i,k}$. We shall discuss this with more details in Section~\ref{sec:avg}.
\end{rem}

As discussed in the introduction, in general one can not expect that the player's actions always converge. However, if the sequence $\{\bm{\alpha}^{n}\}_{n=1}^\infty$ ever admits a limit, denoted by $\bm\alpha^{\infty}$, we expect it to form an open-loop Nash equilibrium under mild assumptions. Intuitively, in the limiting situation, when all other players are using strategies $\alpha_t^{j,\infty}$, $j \neq i$, by some stability argument, player $i$'s optimal strategy to the control problem \eqref{def:J:SFP} should be $\alpha_t^{i,\infty}$, meaning that she will not deviate from $\alpha_t^{i,\infty}$, which makes $(\alpha_t^{i,\infty})_{i=1}^N$ an open loop equilibrium by definition. Therefore, finding an open-loop Nash equilibrium consists of iterating this play until it converges. 

We here give an argument under general problem setup using Pontryagin stochastic maximum principle (SMP). For simplicity, we present the case of uncontrolled volatility without common noise: $\sigma^i(t,\bm{x}, \bm{\alpha}) \equiv \sigma^i(t,\bm{x})$, $\forall i \in \mc{I}$, $\sigma^0 \equiv 0$, and refer to \cite[Chapter 1]{CaDe2:17} for generalization.
The Hamiltonian $H^{i,n+1}: [0,T] \times \Omega \times \RR^{dN} \times \RR^{dN} \times A \hookrightarrow \RR$ for player $i$ at stage $n+1$ is defined by:
\begin{equation}
H^{i,n+1}(t,\omega, \bm x, \bm y, \alpha ) = \bm b(t,\bm x,(\alpha, \bm{\alpha}^{-i,n})) \cdot \bm y + f^i(t,\bm x,(\alpha, \bm{\alpha}^{-i,n})),
\end{equation}
where the dependence on $\omega$ is introduced by $\bm{\alpha}^{-i,n}$. We assume all coefficients $(b^i,\sigma^i, f^i)$ are continuously differentiable with respect to $(\bm x, \bm \alpha) \in \RR^{dN} \times A^N$; $g^i$ is convex and  continuously differentiable with respect to $\bm x \in \RR^{dN}$; $A \in \RR^k$ is convex; the function $H^{i,n+1}$ is convex $\PP$-almost surely in $(\bm x, \bm \alpha)$. By the sufficient part of SMP, we look for a control $\hat\alpha^{i,n+1} \in A$ of the form:
\begin{equation}\label{def:alpha:H}
\hat \alpha^{i,n+1}(t,\omega, \bm x, \bm y) \in \argmin_{\alpha \in A} H^{i,n+1}(t,\omega, \bm x, \bm y, \alpha),
\end{equation}
and solve the resulting forward-backward stochastic differential equations (FBSDEs):
\begin{equation}\label{def:FBSDE}
\left\{ \begin{aligned}
\ud  X_t^{\ell, n+1} &= b^\ell(t, \bm{X}_t^{n+1}, (\hat \alpha^{i,n+1}(t, \bm{X}_t^{n+1}, \bm{Y}_t^{n+1}), \bm{\alpha}_t^{-i,n})) \ud t + \sigma^\ell(t, \bm{X}_t^{n+1}) \ud W_t^\ell, \\
\ud Y_t^{\ell,n+1} &= -\partial_{x^\ell} H^{i,n+1}(t,\bm{X}_t^{n+1}, \bm{Y}_t^{n+1}, \hat \alpha^{i,n+1}(t, \bm{X}_t^{n+1}, \bm{Y}_t^{n+1})) \ud t + \sum_{j=1}^N Z_t^{\ell,j,n+1} \ud W_t^j, \\
 X_0^{\ell,n+1} &= x_0^\ell, \quad Y_T^{\ell,n+1} = \partial_{x^\ell} g^i(\bm X_T^{n+1}), \quad \ell \in \mc{I}.
\end{aligned}
\right.
\end{equation}
If there exists a solution $(\bm{X}^{n+1}, \bm{Y}^{n+1}, \bm{Z}^{n+1}) \in H^2_T(\RR^{dN} \times \RR^{dN} \times \RR^{dN \times mN})$, then an optimal control to problem \eqref{def:J:SFP} is given by plugging the solution into the function $\hat \alpha^{i,n+1}$:
\begin{equation}\label{def:alpha:general}
\alpha^{i,n+1}_t = \hat \alpha^{i,n+1}(t, \bm{X}_t^{n+1}, \bm{Y}_t^{n+1}).
\end{equation}

Now suppose \eqref{def:FBSDE} is solvable, the sequence given in \eqref{def:alpha:general} converges to $\bm\alpha^\infty$ as $n \to \infty$. Denote by $(\bm{X}^{\infty}, \bm{Y}^{\infty}, \bm{Z}^{\infty})$ the solution of \eqref{def:FBSDE} with $\bm{\alpha}^n$ being replaced by $\bm{\alpha}^\infty$. If the system possesses stability, then $(\bm{X}^{\infty}, \bm{Y}^{\infty}, \bm{Z}^{\infty})$ is also the limit of $(\bm{X}^{n+1}, \bm{Y}^{n+1}, \bm{Z}^{n+1}) $. In this case, given other players using $\bm\alpha^{-i,\infty}$, the optimal control of player $i$ is
\begin{equation}
\alpha^{i,\infty}(t, \bm X_t^{\infty}, \bm Y_t^\infty) = \lim_{n\to \infty} \hat \alpha^{i,n}(t, \bm X_t^{n}, \bm Y_t^n) = \lim_{n \to \infty} \alpha^{i,n} = \alpha^{i,\infty},
\end{equation}
where we have used the stability of \eqref{def:FBSDE} and the continuous dependence of $H$ on the parameter $\bm{\alpha}^{-i,n}$ for the first identity, the solvability of \eqref{def:FBSDE} for the second identity, and the convergence of $\alpha^{i,n}$ for the last identity. Therefore, one can put appropriate conditions on $(b^i, \sigma^i, f^i, g^i)$ to ensure these, and we refer to \cite{PeWu:99,PaTa:99,MaMoYo:99,MaWuZhZh:15} for detailed discussions. Remark that, all assumptions are satisfied for the case of linear-quadratic games, and thus all the above arguments can go through. We will give more details in Section~\ref{sec:LQ}.

In general, problem \eqref{def:SFP} is not analytical tractable, and one needs to solve it numerically. Next we present a novel architecture of DNN and a deep learning algorithm that has a parallelization feature. It starts with a brief introduction on deep learning, followed by the detailed deep fictitious play algorithm.

\subsection{Preliminaries on deep learning}\label{sec:nn}
Inspired by neurons in human brains, a neural network (NN) is designed for computers to learn from
observational data. It has become an effective tool in many fields including computer vision, speech recognition, social network filtering, image analysis, {\it etc.}, where results produced by NNs are comparable or even superior to human experts. An example of NNs performing well is image classification, where the task is to identify which of a set of categories a new observation belongs to, on the basis of a training set of data containing observations of known category membership. Denote by $x$ the observations and $z$ its category. This problem consists of efficient and accurate learning of the mapping from observations to categories $x \hookrightarrow z(x)$, which can be complicated and non-trivial. Thanks to the universal approximation theorem and the Kolmogorov-Arnold representation theorem \cite{Cy:89,Ko:91,Ho:91}, NNs are able to provide good approximations to non-trivial mapping.

Our goal is to use deep neural networks to solve the stochastic control problem \eqref{def:SFP}. NNs are made by stacking layers one on top of another. Layers with different functions or neuron structures are called differently, including fully-connected
layer, constitutional layer, pooling layer, recurrent layers, {\it etc.}. As our algorithm \ref{def:algorithm} will focus on fully-connected layers, we here give an example of feed-forward NN using fully-connected layers in Figure \ref{fig:sampleNN}. Nodes in the figure represent neurons and arrows represent the information
flow. As shown, information is constantly “fed forward” from one layer to the next. The first
layer (leftmost column) is called the input layer, and the last layer (rightmost column) is called the output
layer. Layers in between are called hidden layers, as they have no connection with the external world. In
this case, there is only one hidden layer with four neurons.
\begin{figure}[H]
	\centering \includegraphics[width=0.5\textwidth, height = 0.18\textheight]{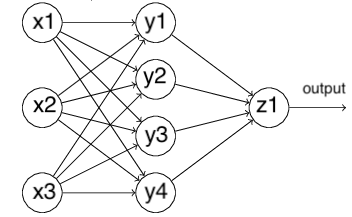}
	\caption{An illustration of a simple feedforward neural network. }\label{fig:sampleNN}
\end{figure}
We now explain how information is processed in NNs. For fully-connected layers, every neuron consists of two kinds of parameters, the weights $w$ and the bias $b$. Each layer can choose an activation function, then an input $x$ goes through it gives $f(w \cdot x + b)$. In the above example of NN, the data $\bm{x} = [x_1, x_2, x_3]$ fed to neuron $y_i$ outputs $f(\bm{w}_j \cdot \bm{x} + b_j)$, $j = 1, \ldots, 4$, which yields $\bm{y} = [y_1,y_2, y_3, y_4]$ as the input of neuron $z_1$. The final output is $z_1 = f(\bm{w}_z \cdot \bm{y} + b_z)$. In traditional classification problems, categorical information $z(\bm{x})$ associated to the input $\bm{x}$ is known, and the optimal weights and bias are chosen to minimize a loss function $L$:
\begin{equation}\label{def:NNloss}
c(w,b) := L(z,z(\bm{x})),
\end{equation}
where $z$ is the output of the NNs, as functions of $(w,b)$, and $z(\bm{x})$ is given from the data. The process of finding optimal parameters is called the training of an NN. 

The activation function $f$ and loss function $L$ are chosen at the user's preference, and common choices are sigmoid $\displaystyle\frac{1}{1+e^{x}}$, ReLU $x^+$ for $f$, and mean squared error $\sum \abs{z - z(\bm x)}^2$ or cross entropy $-\sum z(\bm x) \log(z)$ for $L$ in \eqref{def:NNloss}. In terms of finding the optimal parameters $(\bm{w},\bm{b})$ in \eqref{def:NNloss}, it is in general a high-dimensional optimization problem, and usually done by various stochastic gradient descent methods (e.g. Adam \cite{Adam,Adamcvg}, NADAM \cite{Dozat16}). For further discussions, we refer to \cite[Section 2.1]{Hu:19} and \cite[Section 2.2]{deep:2018}.

However, solving \eqref{def:SFP} is not in line with the above procedure, in the sense that there is no target category $z(\bm{x})$ assigned to each input $\bm{x}$, and consequently, the loss function is not a distance measuring between the network output $z$ and $z(\bm{x})$. We aim at approximating the optimal strategy at each stage by feedforward NNs. What we actually use NN is its ability of approximating complex relations by composition of simple functions (by stacking fully connected layers) and finding the (sub-)optimizer with its well-developed built-in stochastic gradient descent (SGD) solvers. We shall explain further the structures of NNs in the following section.

\subsection{Deep learning algorithms}\label{sec:algorithm}
We introduce the algorithms of deep learning based on fictitious play by describing two key parts as below.
\subsubsection{Part I: solve a stochastic control problem using DNN}\label{sec:DNN}
We in fact solve a time discretization version of problem \eqref{def:SFP}.  Partitioning $[0,T]$ into $N_T$ equally-spaced intervals, with the time step $h = T/N_T$.  Denote by $\tilde {\mathbb{F}} := \{\tilde \MCF_k, 0 \leq k \leq N_T\}$ the ``discretized'' filtration with $\tilde \MCF_k = \sigma\{\bm{W}_{jh}, 0 \leq j \leq k\}$.
A discrete-time analogy of \eqref{def:SFP} is:
\begin{equation}\label{def:sc1}
\tilde \alpha^{i,n+1} =  \argmin_{\{\beta^i_{kh}\in \tilde \MCF_k\}_{k=0}^{N_T-1}}\tilde J^{i}(\beta^i; \tilde{\bm{\alpha}}^{-i,n}),
\end{equation}
where 
\begin{equation}\label{def:sc2}
\tilde J^{i}(\beta^i; \tilde{\bm{\alpha}}^{-i,n}):= \EE\left[\sum_{k=0}^{N_T-1} f^i\left(kh, \bm X_{kh}, (\beta_{kh}^i, \tilde {\bm{\alpha}}_{kh}^{-i,n})\right) h + g^i(\bm X_T)\right],
\end{equation}
and each entry $X_{kh}^\ell$ in $\bm{X}_{kh}$ follows the Euler scheme of \eqref{def:Xt:general} associated to the strategy $\beta^\ell$ if $\ell = i$, and to $\tilde \alpha^{\ell,n}$ if $\ell \neq i$:

\begin{equation}\label{eq:Xt:discrete}
\begin{aligned}
X_{(k+1)h}^\ell &= X_{kh}^\ell + b^\ell(kh, \bm X_{kh}, (\beta_{kh}^i, \tilde{ \bm\alpha}_{kh}^{-i,n})) h + \sigma^\ell(kh, \bm X_{kh}, (\beta_{kh}^i, \tilde{ \bm\alpha}_{kh}^{-i,n}))(W^\ell_{(k+1)h}- W^\ell_{kh}) \\
& \quad + \sigma^0(kh, \bm X_{kh}, (\beta_{kh}^i, \tilde{ \bm\alpha}_{kh}^{-i,n})) (W^0_{(k+1)h}- W^0_{kh}), \quad \ell \in \mc{I}.
\end{aligned}
\end{equation}
Remark that the above time discretization uses Euler scheme, and thus leads to a weak error of $\mathcal{O}({h})$ and a strong error of $\mathcal{O}(\sqrt{h})$.
	
In the discrete setting, $\beta^i_{kh} \in \tilde \MCF_k$ is interpreted as $\beta^i_{kh} = \beta^i_{kh} (\bm{X}_0, \bm{W}_h, \ldots, \bm{W}_{kh})$. Our task is to approximate the functional dependence of the control on noises. Similar to the strategy used in \cite{HaE:16}, we implement this by a multilayer feedforward sub-network:
\begin{equation}\label{def:strategy:approx}
\beta^i_{kh} \sim \beta^i_{kh} (\bm{X}_0, \bm{W}_h, \ldots, \bm{W}_{kh}\vert \theta_{kh}^i),
\end{equation}
where $\theta_{kh}^i$ denotes the collection of all weights and biases in the $k^\text{th}$ sub-network for player $i$. Then, at stage $n+1$, the optimization problem for player $i$ becomes
\begin{equation}\label{def:J:discrete}
\min_{{\left\{\theta_{kh}^{i}\right\}_{k=0}^{N_T-1}} } \EE\left[\sum_{k=0}^{N_T-1} f^i\left(kh, \bm X_{kh}, (\beta_{kh}^i(\theta_{kh}^i), \tilde {\bm{\alpha}}_{kh}^{-i,n})\right) h + g^i\left(\bm X_T\right)\right].
\end{equation}
Denote by $\theta_{kh}^{i,n+1}$ the minimizer of \eqref{def:J:discrete}, then the approximated optimal strategy $\tilde \alpha^{i,n+1}$ is given by \eqref{def:strategy:approx} evaluated at  $\theta_{kh}^{i,n+1}$. Note that even though we only write explicitly the dependence of $\beta^i$'s on $\theta^i$, it affects all $X^i$'s through interactions \eqref{eq:Xt:discrete}. In fact, $X^\ell_{kh}$ depends on $\{\theta_0^{i,n+1}, \ldots, \theta_{(k-1)h}^{i,n+1}\}$,  for all $\ell \in \mc{I}$. Therefore, finding the gradient in minimizing \eqref{def:J:discrete} is a non-trivial task. Thanks to the key feature of NNs, computation can be done via a forward-backward propagation algorithm derived from chain rule composition \cite{Ni:online}. 

The architecture of the NN for finding $\tilde \alpha^{i,n+1}$ is presented in Figure~\ref{fig:algorithm}: ``InputLayer" are inputs of this network; ``Rcost'' and ``Tcost'', representing running and terminal cost, contribute to the total cost $J^i$; ``Sequential'' is a multilayer feedforward subnetwork for control approximation at each time step; ``Concatenate'' is an auxiliary layer combining some of previous layers as inputs of ``Sequential''.

\begin{figure}[H]
	\centering \includegraphics[width=\textwidth]{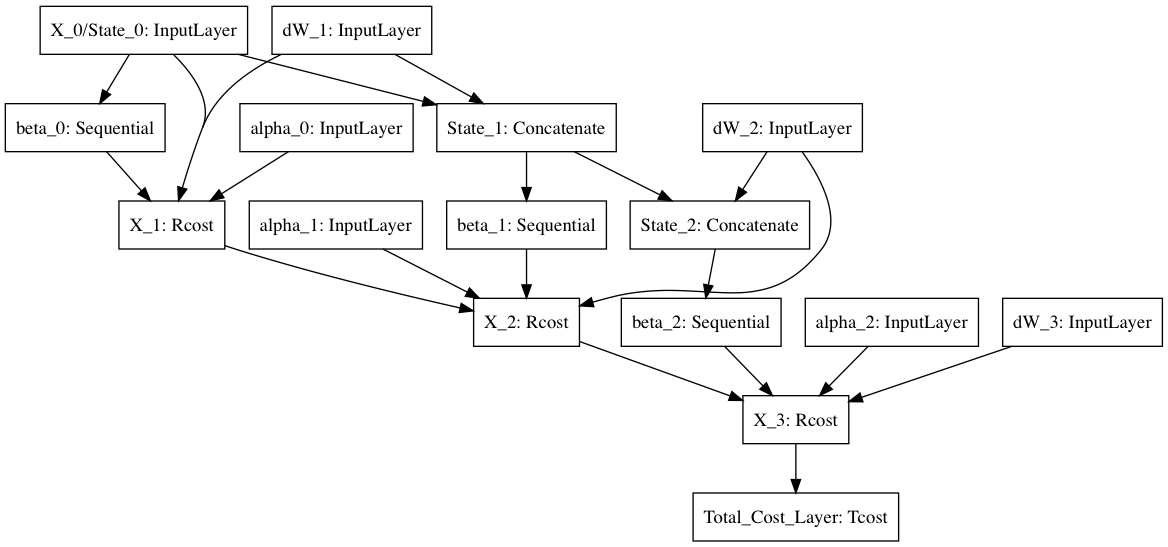}
	\caption{Illustration of the network architecture for problem \eqref{def:J:discrete} with $N_T = T = 3$.}\label{fig:algorithm}
\end{figure}

There are three main kinds of information flows in the network for each period $[kh, (k+1)h]$, $k = 0, \ldots N_T-1$:
\begin{enumerate}
\item $\text{State}_{kh} := (\bm{X}_0, \bm{W}_h, \ldots, \bm{W}_{kh}) \to \beta^i_{kh}$ given by ``Sequential'' layer. It is an $L$-layer feed-forward subnetwork to approximate the control of player $i$ at time $kh$, containing parameters $\theta_{kh}^i$ to be optimized. 

\item $(\bm{X}_{kh}, \beta_{kh}^i, \bm{\alpha}_{kh}^{-i,n}, \ud \bm{W}_{(k+1)h} := \bm{W}_{(k+1)h} - \bm{W}_{kh} ) \to \bm{X}_{(k+1)h}$ given by ``Rcost'' layer. This layer possesses two functions. Firstly, it computes the running cost at time $kh$ using $(\bm{X}_{kh}, \beta_{kh}^i, \tilde{\bm{\alpha}}^{-i,n}_{kh})$, where $\beta_{kh}^i$ is produced from previous step. The cost is then added to the final output. Secondly, it updates states value $\bm{X}_{(k+1)h}$ via dynamics \eqref{eq:Xt:discrete}, using $\beta_{kh}^i$ for player $i$  and using $\bm{\alpha}_{kh}^{-i,n}$ for player $j \neq i$ which are inputs of the network. No parameter is minimized at this layer.

\item $(\text{State}_{kh}, \ud \textbf{W}_{(k+1)h}) \to \text{State}_{(k+1)h}$ given by ``Concatenate'' layer. This layer combines two previous ones together, acting as a preparation for the input of ``Sequential'' layer. No parameter is minimized at this layer. 
\end{enumerate}
At time $T = N_T\times h$, the terminal cost is calculated using $\bm{X}_{T}$ and added to the final output via ``Tcost'' layer. With these preparations, we introduce the deep fictitious play as below. 

\subsubsection{Part II: find an equilibrium by fictitious play}\label{sec:FP}
Here we use a flowchart to describe the algorithm of deep fictitious play (see Algorithm~\ref{def:algorithm}). 

\begin{algorithm}[H]
		\caption{Deep Fictitious Play for Finding Nash Equilibrium \label{def:algorithm}}
	\begin{algorithmic}[1]
		\REQUIRE $N$ = \# of players, $N_T$ = \# of subintervals on $[0,T]$, $M$ = \# of training paths, $M'$ = \# of out-of-sample paths for final evaluation, $\bm{\alpha}^0 = \{\alpha_{kh}^{i,0} \in A \subset \RR^k, i \in \mc{I}\}_{k=0}^{N_T-1}$ = initial belief, $\bm{X}_0 = \{x_0^i \in \RR^d, i \in \mc{I}\}$ = initial states
		  \STATE Create $N$ separated deep neural networks as described in Section~\ref{sec:DNN}
		  \STATE  Generate $M$ sample path of BM: $\bm{W} = \{W_{kh}^{i} \in \RR^m, i \in \mc{I}\cup \{0\} \}_{k=1}^{N_T} $
		  \STATE $n \gets 0$
		  \REPEAT

		  \STATE $n \gets n+1$
		  \FOR{$i \gets 1$ to $N$} 
		  \STATE (Continue to) Train $i^{th}$ NN with data $\{\bm{X}_0, \bm\alpha^{-i,n-1} = \{\alpha_{kh}^{j,n-1}, j \in \mc{I}\setminus \{i\}\}_{k=0}^{N_T-1}, \bm W\}$
		  \STATE Obtain the approximated optimal strategy  $\alpha^{i,n}$ and cost $J^i(\alpha^{i,n}; \bm\alpha^{-i,n-1})$ 
		  \ENDFOR
		  \STATE Collect optimal policies at stage $n$: $\bm{\alpha}^n \gets (\alpha^{1,n}, \ldots, \alpha^{N,n})$
		  \STATE \label{algo:step}Compute relative change of cost $\displaystyle err^n := \max_{i \in \mc{I}}\left\{\frac{\abs{J^i(\alpha^{i,n}; \bm\alpha^{-i,n-1}) - J^i(\alpha^{i,n-1}; \bm{\alpha}^{-i,n-2})}}{ J^i(\alpha^{i,n-1}; \bm{\alpha}^{-i,n-2})}\right\}$

		  \UNTIL$err^n$ go below a threshold

		\STATE Generate $M'$ out-of-sample paths of BM for final evaluation
		\STATE $n' \gets 0$
		\REPEAT 
		\STATE $n' \gets n'+1$
		\STATE Evaluate $i^{th}$ NN with  \{$\bm X_0$, $\bm{\alpha}^{-i,n'-1}$, out-of-sample paths\}, $\forall i \in \mc{I}$
		\STATE Obtain $\alpha^{i,n'}$ and $J^{i,n'} := J^i(\alpha^{i,n'};\bm{\alpha}^{-i, n'-1})$ $\forall i \in \mc{I}$
		\UNTIL $J^{i,n'}$ converges in $n'$, $\forall i \in \mc{I}$
		\RETURN The optimal policy $\alpha^{i,n'}$, and the final cost for each player $J^{i,n'}$
	\end{algorithmic}
\end{algorithm}

\subsection{Implementation}

{\it Computing environment.} The Algorithm~\ref{def:algorithm} described in Section~\ref{sec:FP} is implemented in Python using the high-level neural network API Keras \cite{Keras}. Numerical examples will be presented in Section~\ref{sec:numerics}. All experiments are performed using Amazon EC2 services, which provide a variety of instances for computing acceleration. All computations use NVIDIA K80 GPUs with 12GiB of GPU memory on Deep Learning Amazon Machine Image running on Ubuntu 16.04.

{\it Parallelizability.} As $N$ going relatively large, to make computation manageable, one can distribute Step $5-9$ to several GPUs. That is, assigning each available GPU the task of training a subset of neural networks, where this subset is fixed from stage to stage. This will speed up the computation time significantly, as peer-to-peer GPU communications are not needed in the designed algorithm. 

{\it Input, output and parameters for neural networks.} Before training, we sample $\bm{W} = \{W_{kh}^{i} \in \RR^m, i \in \mc{I} \}_{k=1}^{N_T}$, which, together with the initial states $\bm{X}_0$ and initial belief $\bm{\alpha}^0 = \{\alpha_{kh}^{i,0} \in A \subset \RR^k, i \in \mc{I}\}_{k=0}^{N_T-1}$, are the inputs of NNs. Adam, a variant of SGD that adaptively estimates lower-order moments, is chosen to optimize the parameters $\{\theta_{kh}^i\}_{k=0}^{N_T-1}$. The hyper-parameters set for Adam solver follows the original paper \cite{Adam}. Regarding the architecture of ``Sequential'', it is a $L$-layered subnetwork. We set $L=4$, with 1 input layer, 2 hidden layers, and 1 output layer containing $k$ nodes. Rectified linear unit is chosen for hidden layers while no activation is applied to the output layer. We also add Batch Normalization \cite{batch} for hidden layers before activation. This method performs the normalization for each training mini-batch to eliminate  internal covariate shift phenomenon, and thus frees us from delicate parameter initialization. It also acts as a regularizer, in some cases eliminating the need for Dropout. Note that the choice of $L$ and size of $\{\theta_{kh}^i\}_{k=0}^{N_T-1}$ are empirical. For testing problems that have benchmark solutions, one can do grid-search method to select the one with the best performance in the validation set. However, for real problems there is no universal rule for all problem settings.

Parameters of the network are initialized at Step 1. In Step 7, training continues from previous stage without re-initialization. This is because, although opponents' policies change from stage to stage, they will not vary significantly and parameter values from previous stage should be better than a random initialization. For fixed computational budget, instead of using the stopping criteria in Step 12 one can terminate the loop until $n$ reaches a predetermined upper bound $\bar n$. In Step 7, the number of epochs to train the model at every single stage does not need to be large (at the scale of hundreds). This is because we are not aiming at a one-time accurate approximation of the optimal policy. Especially at the first few rounds when opponents' policies are far from optimal, pursuing accurate approximation is not meaningful. Instead, by using small budget to obtain moderate accuracy at each iteration, we are able to repeat the game for more times. In summary, for the two computational scheme: large $\bar n$ small epochs, and small $\bar n$ large epochs, the former one is better. 

If opponents' policies stay the same from stage to stage, then the two schemes receive the same accuracy. This is justified by the following argument: Suppose opponents' policies stay the same, then player $i$ essentially faces the same optimization problem from stage to stage. Since we do not re-initialize network parameters in Step 7, the difference between the two schemes is to train the same problem with small epochs and large rounds \emph{vs.} large epochs and small rounds. This is the same in terms of SGD training, thus should lead to the same relative error.
In reality, the opponents' policies is updated from time to time, and the former scheme enables us to obtain player $i$'s reaction with more updated belief of his opponents. Step 15-19 are not computational costly, and the value functions usually converge after several iterations in our numerical study.

 \section{Linear-Quadratic games}\label{sec:LQ}
Although the deep fictitious theory and algorithm can be applied for any $N$-player game, the proof of convergence is in general hard. Here 
we consider a special case of linear-quadratic symmetric $N$-player games, and analyze the convergence of $\bm{\alpha}^{n}$ defined in \eqref{def:SFP}. 
The strategy analyzed here will provide an open-loop Nash equilibrium, as proved at the end of section. 

We follow the linear-quadratic model proposed in \cite{CaFoSu:15}, where players's dynamics interact through their empirical mean:
\begin{equation}
\ud X_t^i = [a(\overline X_t - X_t^i)   + \alpha_t^i ] \ud t  + \sigma \left( \rho \ud W_t^0 + \sqrt{1-\rho^2} \ud W_t^i\right), \quad X_0^i = x^i, \quad \overline X_t = \frac{1}{N}\sum_{i=1}^N X_t^i. \label{def:Xt}
\end{equation}
Here $\{W_t^i, 0 \leq i \leq N\}$ are independent standard Brownian motions (BMs). 
Each player $i \in \{1, 2, \ldots, N\}$ controls the drift by $\alpha_t^i$ in order to minimize the cost functional
\begin{equation}
J^i(\alpha^1, \ldots, \alpha^N) = \EE\left\{ \int_0^T f^i(\bm{X}_t, \alpha_t^i) \ud t + g^i(\bm{X}_T)\right\}, \label{def:J}
\end{equation}
with the running cost defined by 
\begin{equation}
f^i(\bm{x},\alpha) = \half \alpha^2 - q \alpha(\bar x - x^i) + \frac{\eps}{2}(\bar x - x^i)^2, \quad \bar x = \frac{1}{N}\sum_{i=1}^N x^i,
\end{equation}
and the terminal cost function $g^i$ by 
\begin{equation}
g^i(\bm{x}) = \frac{c}{2}(\bar x - x^i)^2. 
\end{equation}
All parameters $a, \eps, c, q$ are non-negative, and $q^2 \leq \eps$ is imposed so that $f^i(\bm{x}, \alpha)$ is convex in $(\bm{x}, a)$. In \cite{CaFoSu:15}, $X_t^i$ is viewed as the log-monetary reserves of bank $i$ at time $t$. For further interpretation, we refer to \cite{CaFoSu:15}. 

In the spirit of fictitious play, the $N$-player game is recasted into $N$ individual optimal control problems played iteratively. The players start with a smooth belief of their opponents' actions $\bm{\alpha}^0$. At stage $n+1$, the players have observed the same past controls $\alphain$'s, and then each player optimizes her control problem individually, assuming other players will follow their choice at state $n$. 
That is, for player $i$'s problem, her dynamics are controlled through $\alpha^i_t$, while other players' states evolve according to the past strategies $\bm{\alpha}^{-i,n}$:
\begin{align}
\ud X_t^{i,n+1} &= [a(\overline X_t^{n+1} - X_t^{i,n+1}) + \alpha_t^i] \ud t + \sigma (\rho \ud W_t^0 + \sqrt{1-\rho^2} \ud W_t^i),\label{def:Xti:lq} \\
\ud X_t^{j, n+1} & = [a(\overline X_t^{n+1} - X_t^{j, n+1}) + \alpha_t^{j,n}] \ud t + \sigma (\rho \ud W_t^0 + \sqrt{1-\rho^2} \ud W_t^j), \quad j \neq i. \label{def:Xtj:lq}
\end{align}
Player $i$ faces an optimal control problem: 
\begin{equation}\label{def:J:lq}
\begin{aligned}
&\inf_{\alpha^i \in \mbA} J^{i,n+1}(\alpha^i; \bm{\alpha}^{-i,n}), \text{ where } \\
& J^{i,n+1}(\alpha^i;\bm{\alpha}^{-i,n}) := \EE\left\{\int_0^T   \half (\alpha_t^i)^2 - q\alpha_t^i (\overline X_t^{n+1} - X_t^{i,n+1}) +  \frac{\eps}{2}(\overline X_t^{n+1} - X_t^{i,n+1})^2 \ud t \right.\\
&\hspace{150pt}+ \frac{c}{2}(\overline X_T^{n+1} - X_T^{i,n+1})^2\Bigg\}.
 \end{aligned}
\end{equation}
The space where we search for optimal $\alpha^i$ is the space of square-integrable progressively-measurable $\RR$-valued processes on $\mbA: = \mathbb{H}^2_T(\RR)$, to be consistent with open-loop equilibria. Denote by  $\alpha^{i,n+1}$ the minimizer of this control problem at stage $n+1$:
\begin{equation}\label{def:alphaast}
\alpha^{i,n+1} := \argmin_{\alpha^i \in \mbA} J^{i,n+1}(\alpha^i; \bm{\alpha}^{-i,n}).
\end{equation}
In what follows, we shall show:
\begin{enumerate}[(a)]
\item $\alpha^{i,n+1}$ exists $\forall i \in \mc{I}, n \in \NN$, that is, the minimal cost in \eqref{def:J:lq} is always attainable;
\item the family $\{\bm{\alpha}^n\}$ converges;
\item the limit of $\bm\alpha^n$ forms an open-loop Nash equilibrium.
\end{enumerate}

\subsection{The probabilistic approach}
Observing that the cost functional $J^{i,n+1}$ in \eqref{def:J:lq} solely depends on the process $\widetilde X^{i,n+1} := \overline X^{n+1}  - X^{i,n+1}$ and the control $\alpha^i$, we make the following simplification.  Notice that \eqref{def:Xti:lq} and \eqref{def:Xtj:lq} imply
\begin{equation}\label{def:Xttilde}
\ud \widetilde X_t^{i, n+1} =\left[\frac{\sum_{j\neq i} \alpha_t^{j,n}}{N} - \frac{N-1}{N}\alpha_t^i - a \widetilde X_t^{i,n+1}\right] \ud t + \sigma\sqrt{1-\rho^2}(\frac{1}{N}\sum_{j=1}^N \ud W_t^j - \ud W_t^i).
\end{equation}
Then, player $i$'s problem is equivalent to:
\begin{equation}
\inf_{\alpha^i \in \mbA}\EE\left\{\int_0^T   \half (\alpha_t^i)^2 - q\alpha_t^i \widetilde X_t^{i, n+1} +  \frac{\eps}{2} (\widetilde X_t^{i,n+1})^2 \ud t + \frac{c}{2}(\widetilde X_T^{i,n+1})^2\right\}.
\end{equation}

In what follows, we show the existence of unique minimizer, denoted by $\alpha^{i,n+1}$, using SMP. The Hamiltonian for player $i$ at stage $n+1$ reads as
\begin{equation}
H^{i,n+1}(t,\omega, x,y,\alpha) = (\frac{\sum_{j\neq i} \alpha^{j,n}_t}{N}-\frac{N-1}{N}\alpha - ax)y + \half \alpha^2 - q\alpha x + \frac{\eps}{2}x^2.
\end{equation}
For a given admissible control $\alpha^i \in \mbA$, the adjoint processes $(Y_t^{i,n+1}, Z_t^{i,j,n+1}, 0 \leq j \leq N)$ satisfy the backward stochastic differential equation (BSDE):
 \begin{equation}\label{def:Yt}
 \ud Y_t^{i,n+1} = -[-aY_t^{i,n+1} - q\alpha_t^i + \eps \widetilde X_t^{i,n+1}]\ud t  + \sum_{j=0}^N Z_t^{i,j,n+1} \ud W_t^j,
  \end{equation}
with the terminal condition $Y_T^{i,n+1} = c\widetilde X_T^{i,n+1}$. Standard results on BSDE \cite{PaPe:90}, together with the estimates on the controlled state $\widetilde X_t^{i,n+1}$, guarantee the existence and uniqueness of adjoint processes. Pontryagin SMP suggests the form of optimizer:
\begin{equation}\label{def:alpha}
\partial_\alpha H^{i,n+1} = 0 \iff \hat\alpha = qx + \frac{N-1}{N}y.
\end{equation}
Plugging this candidate into the system \eqref{def:Xttilde}-\eqref{def:Yt} produces a system of affine FBSDEs:
\begin{equation}\label{eq:FBSDE}
\left\{ \begin{aligned}
\ud \widetilde X_t^{i, n+1} &=\left[\frac{\sum_{j\neq i} \alpha_t^{j,n}}{N} - (a + (1-\oon)q) \widetilde X_t^{i,n+1} - (1-\oon)^2 Y_t^{i,n+1}\right] \ud t \\
&\qquad  + \sigma\sqrt{1-\rho^2}(\frac{1}{N}\sum_{j=1}^N \ud W_t^j - \ud W_t^i), \\
 \ud Y_t^{i,n+1} &= -[-(a + (1-\oon)q)Y_t^{i,n+1}  +  (\eps-q^2) \widetilde X_t^{i,n+1}]\ud t  + \sum_{j=0}^N Z_t^{i,j,n+1} \ud W_t^j, \\
 \widetilde X_0^{i,n+1} &= \overline x_0 - x_0^i, \quad Y_T^{i,n+1} = c\widetilde X_T^{i,n+1}.
\end{aligned}
\right.
\end{equation}
The sufficient condition of SMP suggests that if we solves \eqref{eq:FBSDE}, we actually have obtained the optimal control by plugging its solution into equation \eqref{def:alpha}. In fact, the coefficients satisfy the $G$-monotone property in \cite{PeWu:99}, thus the system is uniquely solved in $\mathbb{H}^2_T(\RR \times \RR \times \RR^{N+1}) $, and the resulted optimal control is indeed admissible. This answers question (a). For the other two questions, we need to further analyze \eqref{eq:FBSDE}. 

Note that the system can be decoupled using:
\begin{equation}\label{def:Ytansatz}
Y_t^{i,n+1} = K_t \widetilde X_t^{i,n+1} - \psi_t^{i,n+1},
\end{equation}
where $K_t$ satisfies the Riccati equation:
\begin{equation}\label{def:Kt}
\dot K_t = 2(a  + (1-\frac{1}{N})q) K_t + (\frac{N-1}{N})^2 K_t^2 - (\eps - q^2), \quad K_T = c,
\end{equation}
and the decoupled processes $(\widetilde X_t^{i,n+1}, \psi_t^{i,n+1}, \phi_t^{i,j,n+1}, 0 \leq j \leq N)$ satisfy:
\begin{equation}\label{eq:FBSDE:decouple}
\left\{ \begin{aligned}
\ud \widetilde X_t^{i, n+1} &=\left[\frac{\sum_{j\neq i} \alpha_t^{j,n}}{N} - \gamma_t \widetilde X_t^{i,n+1} + (1-\oon)^2 \psi_t^{i,n+1}\right] \ud t \\
& \qquad + \sigma\sqrt{1-\rho^2}(\frac{1}{N}\sum_{j=1}^N \ud W_t^j - \ud W_t^i), \\
\ud \psi_t^{i,n+1} &= -[-\gamma_t \psi_t^{i,n+1} - K_t \frac{\sum_{j\neq i} \alpha_t^{j,n}}{N}]\ud t + \sum_{j=0}^N \phi_t^{i,j,n+1} \ud W_t^j, \\
\widetilde X_0^{i,n+1} &= \overline x_0 - x_0^i, \quad \psi_T^{i,n+1} = 0,
\end{aligned}
\right.
\end{equation}
where $\gamma_t$ is a deterministic function on $[0,T]$:
\begin{equation}\label{def:gamma}
\gamma_t = a + (1-\oon)q + (1-\oon)^2K_t,
\end{equation}
and the optimal strategy is expressed as
\begin{equation}\label{eq:alpha}
\alpha_t^{i,n+1} = (q + (1-\oon)K_t) \widetilde X_t^{i,n+1} - (1-\oon)\psi_t^{i,n+1}.
\end{equation}
Again, since $\bm\alpha^n \in \mathbb{H}_T^2(\RR^N)$, existence and uniqueness of $(\psi^{i,n+1}, \phi^{i,j,n+1}, 0\leq j \leq N) \in \mathbb{H}^2(\RR \times \RR^{N+1})$  is guaranteed $\forall i \in \mc{I}$, $n \in \NN$, and the forward equation possesses a unique strong solution. 
Then the triple $(X^{i,n+1},Y^{i,n+1},Z^{i,j,n+1})$ solves the original FBSDEs \eqref{eq:FBSDE} with $Y_t^{i,n+1}$ defined by \eqref{def:Ytansatz} and $Z_t^{i,j, n+1}$ by 
\begin{align}
Z_t^{i,0,n+1} = -\phi_t^{i,0,n+1}, \quad Z_t^{i,j,n+1} = -\phi_t^{i,j,n+1} + K_t\sigma \sqrt{1-\rho^2}(\frac{1}{N} - \delta_{i,j}), \quad j \in \mc{I}.
\end{align}

To answer questions (b) and (c), we state the main theorem in this section, with the proofs  presented in the next subsections.

\begin{thm}\label{thm:cvg}
	For linear-quadratic games, the family $\{\bm \alpha^n\}_{n \in \NN}$ defined in \eqref{def:J:lq}-\eqref{def:alphaast} converges if
	\begin{equation}\label{def:cvgcondition}
	\frac{1-e^{-2T\underline \gamma}}{\underline \gamma}C < 1.
	\end{equation}
	It forms an open-loop Nash equilibrium of the original problem \eqref{def:Xt}-\eqref{def:J}. Moreover, the limit, denote by $\bm\alpha^\infty$, is independent from the choice of initial belief $\bm\alpha^0$. Here  $\underline \gamma = a + (1-\oon)q + (1-\oon)^2 \underline K$, $\overline K$ and $\underline K$ are the maximum and minimum value of $K_t$ on $[0,T]$, and the constant C is 
	\begin{equation}\label{def:C}
	C = (1-\oon)^2 \left((1-\oon)^2 \overline K^2 + (q + (1-\oon)\overline K)^2 \left(\frac{1-e^{-2T\underline \gamma}}{\underline \gamma}(1-\oon)^4\overline K^2 + 2\right)\right).
	\end{equation}
\end{thm}

\begin{remark}
The condition \eqref{def:cvgcondition} is sufficient but not necessary. The numerical performance of the proposed algorithm can do better. In Section~\ref{sec:numerics}, the parameters are chosen so that the condition is violated, but the algorithm still converges fast, in order to illustrate the sufficiency. By observing the form of $C$ and $\underline \gamma$, we remark that the convergence rate decreases in the number of players $N$.
\end{remark}
\begin{proposition}\label{prop:case}
	The following three classes of parameters satisfy condition \eqref{def:cvgcondition}:
	\begin{enumerate}[(i)]
		\item\label{case:t} Small time duration, that is, $T$ is small.
		\item\label{case:a} Strong mean-reversion rate, {\it i.e.}, $a$ is large.
		\item\label{case:c} Small terminal cost and small intensive to borrowing or landing, that is, $c$ and $q$ are small. Also the ``remaining'' running cost of the state process\footnote{The running cost $f^i(\bm{x}, \alpha)$ can be rewritten as $f^i(\bm x, \alpha) = \half (\alpha - q(\bar x  - x^i))^2 + \half (\eps - q^2)(\bar x - x^i)^2$, therefore, can be interpreted as penalizing the control from deviating $q(\bar x - x^i)$, borrowing or lending proportionally to the difference from average with a rate $q$, as well as penalizing the distance from average with weight $\eps - q^2$. } is small, i.e., $\eps - q^2$ is small.
	\end{enumerate}
\end{proposition}

\begin{proof}
We first notice that  the solution to \eqref{def:Kt} is smooth and monotone on $[0,T]$, by computing its derivative:
\begin{equation}
\dot K_t \sim -(\eps - q^2) + c^2(1-\oon)^2 + 2c(a + (1-\oon)q).
\end{equation}
So $\overline K = \max\{c, K_0\}$ and $\underline K = \min\{c, K_0\}$. Also, when $\dot K_t >0$, $K_0$ is bounded below by $\frac{-(\eps - q^2)-c\delta^+}{\delta^- - c(1-\oon)^2}$; otherwise when $K_t$ is decreasing,  $K_0$ is bounded above by $\frac{-(\eps - q^2)-c\delta^+}{\delta^- - c(1-\oon)^2}$, where 
\begin{equation}
\delta^\pm = -(a + (1-\oon)q) \pm \sqrt R, \quad  R = (a + (1-\oon)q)^2 + (1-\oon)^2 (\eps - q^2).
\end{equation}
Then case \eqref{case:t} follows by the fact that $C$ has an upper bound that is free of $T$.

For $a$ sufficiently large, $K_t$ is increasing and $\overline K = c$. Then $C$ has a upper bound (uniformly in $a$), and case \eqref{case:a} follows $\frac{1 - e^{-2T\underline \gamma}}{\underline \gamma} < \frac{1}{a}$. Under case \eqref{case:c}, $\overline K$ is sufficiently small, thus $C$ is small and the factor is less than 1. 
\end{proof}

\subsection{Proof of convergence}
This section proves Theorem~\ref{thm:cvg}.  
Define $\Delta \zeta_t^{i,n} := \zeta_t^{i,n+1} - \zeta_t^{i,n}$ the difference from stage $n$ to $n+1$ for the $i^{th}$ player, with $\zeta = \alpha, \psi, \phi, \widetilde X$, respectively.
Using equation \eqref{eq:FBSDE:decouple}, the increment in $\psi$ satisfies:
\begin{align}
\ud \Delta \psi_t^{i,n} = -[-\gamma_t \Delta \psi_t^{i,n} - \frac{K_t}{N} \sum_{j \neq i} \Delta \alpha_t^{j,n-1}] \ud t + \sum_{j=0}^N \Delta \phi_t^{i,j,n}\ud W_t^j,  \quad \Delta \psi_T^{i,n} = 0,
\end{align}
whose solution is:
\begin{equation}
\Delta \psi_t^{i,n}  = \EE\left[ \int_t^T - \frac{K_s}{N}\sum_{j \neq i} \Delta \alpha_s^{j,n-1} e^{\int_s^t \gamma_u \ud u} \ud s \Bigg\vert \MCF_t\right].
\end{equation}
By Jensen's inequality, one deduces:
\begin{align}
\ltwonorm{\Delta \psi^{i,n}}^2 &\leq \int_0^T \EE\left[\int_t^T \frac{K_s^2}{N^2}\left(\sum_{j \neq i} \Delta \alpha_s^{j,n-1}\right)^2 e^{2\int_s^t \gamma_u \ud u} \ud s\right] \ud t  \\
&\leq \frac{\overline{K}^2}{N^2} \int_0^T \int_t^T \EE\left(\sum_{j \neq i} \Delta \alpha_s^{j,n-1}\right)^2 e^{2(t-s)\underline{\gamma}} \ud s \ud t\\
& = \frac{\overline{K}^2}{N^2} \int_0^T  \EE\left(\sum_{j \neq i} \Delta \alpha_s^{j,n-1}\right)^2 \frac{1-e^{-2s\underline{\gamma}}}{2\underline{\gamma}} \ud s  \\
&\leq  \frac{\overline{K}^2}{N^2}\frac{1-e^{-2T\underline{\gamma}}}{2\underline{\gamma}} (N-1)^2 \max_{j \neq i } \int_0^T \EE[\Delta \alpha_s^{j,n-1}]^2 \ud s \\
& \leq \frac{1-e^{-2T\underline{\gamma}}}{2\underline{\gamma}}(1-\oon)^2 \overline{K}^2 \max_{i \in \mc{I}} \ltwonorm{\Delta\alpha^{i,n-1}}^2,
\end{align}
where $\underline{\gamma} = a + (1-\oon)q + (1-\oon)^2 \underline K$. Since the RHS of the above inequality is independent of $i$, taking maximum over $\mc{I}$ yields
\begin{equation}\label{eq:psi:ineq}
\max_{i \in \mc{I}} \ltwonorm{\Delta \psi^{i,n}}^2 \leq \frac{1-e^{-2T\underline{\gamma}}}{2\underline{\gamma}}(1-\oon)^2 \overline{K}^2 \max_{i \in \mc{I}} \ltwonorm{\Delta\alpha^{i,n-1}}^2.
\end{equation}

Similarly, the dynamics of $\Delta \widetilde X_t^{i,n}$ can be derived from \eqref{eq:FBSDE:decouple}:
\begin{equation}
\ud \Delta \widetilde X_t^{i,n} = [\oon \sum_{j \neq i} \Delta\alpha_t^{j,n-1} - \gamma_t \Delta \widetilde X_t^{i,n} + (1-\oon)^2 \Delta \psi_t^{i,n}]\ud t, \quad \Delta\widetilde X_0^{i,n} = 0,
\end{equation}
which admits the solution:
\begin{equation}
\Delta \widetilde X_t^{i,n} = \int_0^t \left(\oon \sum_{j \neq i} \Delta\alpha_s^{j,n-1} + (1-\oon)^2 \Delta \psi_s^{i,n}\right) e^{-\int_s^t \gamma_u \ud u}\ud s.
\end{equation}
We next give an upper bound of increment of the forward process $\Delta \widetilde X_\cdot^{i,n}$:
\begin{align}
\ltwonorm{\Delta \widetilde X^{i,n}}^2 &\leq \int_0^T \int_0^t  \EE \left(\oon \sum_{j \neq i} \Delta\alpha_s^{j,n-1} + (1-\oon)^2 \Delta \psi_s^{i,n}\right)^2   e^{-2\int_s^t \gamma_u \ud u}\ud s \ud t \\
&  \leq 2\int_0^T \int_0^t \left( \EE [\oon \sum_{j \neq i} \Delta\alpha_s^{j,n-1}]^2 + (1-\oon)^4 \EE[ \Delta \psi_s^{i,n}]^2\right)   e^{-2(t-s)\underline\gamma} \ud s \ud t \\
& \leq 2\int_0^T \left( \EE [\oon \sum_{j \neq i} \Delta\alpha_s^{j,n-1}]^2 + (1-\oon)^4 \EE[ \Delta \psi_s^{i,n}]^2\right)  \frac{1-e^{-2(T-s)\underline \gamma}}{2\underline \gamma} \ud s \\
& \leq \frac{1-e^{-2T\underline \gamma}}{\underline \gamma} \left( (1-\oon)^2 \max_{j \neq i} \ltwonorm{\Delta \alpha^{j,n-1}}^2 + (1-\oon)^4 \ltwonorm{\Delta \psi^{i,n}}^2\right).
\end{align}
Again by taking maximum over $\mc{I}$ on both sides, one has:
\begin{equation}\label{eq:Xttilde:ineq}
\max_{i \in \mc{I}} \ltwonorm{\Delta \widetilde X^{i,n}}^2 \leq  \frac{1-e^{-2T\underline \gamma}}{\underline \gamma} \left( (1-\oon)^2 \max_{i \in \mc{I}} \ltwonorm{\Delta \alpha^{i,n-1}}^2 + (1-\oon)^4 \max_{i \in \mc{I}} \ltwonorm{\Delta \psi^{i,n}}^2\right).
\end{equation}

Recall from \eqref{eq:alpha} that the increment in the strategy can be decomposed as
\begin{equation}
\Delta \alpha_t^{i,n} = (q + (1-\oon)K_t) \Delta \widetilde X_t^{i,n} - (1-\oon) \Delta \psi_t^{i,n},
\end{equation}
together with estimates \eqref{eq:psi:ineq} and \eqref{eq:Xttilde:ineq}, we obtain:
\begin{align}
\max_{i \in \mc{I}} \ltwonorm{\Delta \alpha^{i,n}}^2 & \leq 2(q + (1-\oon)\overline K)^2 \max_{i \in I} \ltwonorm{\Delta \widetilde X^{i,n}}^2  + 2(1-\oon)^2 \max_{i \in \mc{I}} \ltwonorm{\Delta \psi^{i,n}}^2\\
& \leq \frac{1-e^{-2T\underline{\gamma}}}{\underline{\gamma}} C  \max_{i \in \mc{I}} \ltwonorm{\Delta \alpha^{i,n-1}}^2,
\end{align}
where $C$ is a constant given in \eqref{def:C}. 
Under condition \eqref{def:cvgcondition}, the mapping $\Delta \bm{\alpha}^{n-1} \hookrightarrow \Delta\bm{\alpha}^{n}$ is a contraction. Therefore, this proposed learning process converges in the linear-quadratic games. 

Denote the limit of $\{\bm{\alpha}^n\}$ by $\bm\alpha^\infty = [\alpha^{1,\infty}, \ldots, \alpha^{N,\infty} ]$ where the learning process start with an initial belief $\bm{\alpha}^0$. Let $(\widetilde X_t^{i,\alpha}, \psi_t^{i,\alpha}, \phi_t^{i,\alpha})$ be the solution to the decoupled system \eqref{eq:FBSDE:decouple} with $\{{\alpha}^{j,n}, j \in \mc{I}\setminus\{i\}\}$ replaced by $\{{\alpha}^{j,\infty}, j \in \mc{I}\setminus\{i\}\}$. On one hand, this corresponds to the problem of identifying player $i$'s best strategy, while others using $\bm\alpha^{-i,\infty}$, and her best choice is 
\begin{equation}
(q + (1-\oon)K_t) \widetilde X_t^{i,\alpha} - (1-\oon)\psi_t^{i,\alpha}.
\end{equation}
On the other hand, by stability theorems (e.g. \cite[Theorem 3.4.2, Theorem 4.4.3]{Zh:17}), this triple $(\widetilde X_t^{i,\alpha}, \psi_t^{i,\alpha}, \phi_t^{i,\alpha})$ is also the $L^2$ limit of $(\widetilde X_t^{i,n}, \psi_t^{i,n}, \phi_t^{i,n})$. Therefore, letting $n \to \infty$ in equation \eqref{eq:alpha} gives
\begin{equation}\label{eq:lmtrelation}
\alpha^{i,\infty} = (q + (1-\oon)K_t) \widetilde X_t^{i,\alpha} - (1-\oon)\psi_t^{i,\alpha}.
\end{equation}
Therefore, the best response for player $i$ is $\alpha^{i,\infty}$, given others play $\bm{\alpha}^{-i,\infty}$, indicating that the limit $\bm\alpha^\infty$ forms an open-loop Nash equilibrium.

It remains to prove  that the limit is independent from the initial belief. Suppose that there exist two limits $\bm{\alpha}^\infty $ and $\bm{\beta}^\infty$ arisen from two distinguished initial belief $\bm{\alpha}^0$ and $\bm{\beta}^0$, and let $(\widetilde X_t^{i,\beta}, \psi_t^{i,\beta}, \phi_t^{i,\beta})$ be the solution to \eqref{eq:FBSDE:decouple} associated with $\bm{\beta}^\infty$.
Following similar derivations in the proof of convergence gives:
\begin{align}
&\max_{i\in\mc{I}}\ltwonorm{\psi^{i,\alpha} - \psi^{i,\beta}}^2 \leq \frac{1-e^{-2T\underline{\gamma}}}{2\underline{\gamma}}(1-\oon)^2 \overline{K}^2 \max_{i \in \mc{I}} \ltwonorm{\alpha^{i,\infty} - \beta^{i,\infty}}^2, \\
&\max_{i \in \mc{I}} \ltwonorm{ \widetilde X^{i,\alpha} - \widetilde X^{i,\beta}}^2 \nonumber \\
&\qquad \leq  \frac{1-e^{-2T\underline \gamma}}{\underline \gamma} \left( (1-\oon)^2 \max_{i \in \mc{I}} \ltwonorm{\alpha^{i,\infty} - \beta^{i,\infty}}^2 + (1-\oon)^4 \max_{i \in \mc{I}} \ltwonorm{\psi^{i,\alpha} - \psi^{i,\beta}}^2\right).
\end{align}
Combining the above equations together, and using \eqref{eq:lmtrelation} for both $\alpha^{i,\infty}$ and $\beta^{i,\infty}$, 
 we deduce:
\begin{equation}
\max_{i \in \mc{I}}\ltwonorm{ \alpha^{i,\infty} - \beta^{i,\infty}}^2  \leq  \frac{1-e^{-2T\underline{\gamma}}}{\underline{\gamma}} C  \max_{i \in \mc{I}} \ltwonorm{ \alpha^{i,\infty} - \beta^{i,\infty}}^2. 
\end{equation}
Under the same condition \eqref{def:cvgcondition},  $\bm{\alpha}^\infty = \bm{\beta}^\infty$ in the $L^2$ sense. Therefore, we have shown that, independent of initial belief, the fictitious play will converge and the limit is unique.  

\subsection{Identifying the limit}
 As proved in Theorem~\ref{thm:cvg}, the limiting strategy $\bm\alpha^\infty$ forms an open-loop Nash equilibrium, and in this section, we verify it coincides with the equilibrium provided in \cite{CaFoSu:15} by direct calculations. 

Recall from \cite{CaFoSu:15}, the open-loop Nash equilibrium to the original $N$-player problem \eqref{def:Xt}--\eqref{def:J} is:
\begin{equation}\label{def:OLE}
\alpha_t^{i,\ast} = [q + (1-\frac{1}{N})\eta_t] (\overline X_t^\ast- X_t^{i,\ast}),
\end{equation}
where $X_t^{i,\ast}$ is the solution to \eqref{def:Xt} associated with $\alpha_t^{i,\ast}$, $\overline X_t^\ast$ is the average of $X_t^{i,\ast}$, and $\eta_t$ solves a Riccati equation:
\begin{equation}\label{def:eta}
\dot \eta_t = 2(a + (1-\frac{1}{2N})q)\eta_t + (1-\frac{1}{N})\eta_t^2 - (\eps - q^2), \quad \eta_T = c.
\end{equation}
Note that, the expression \eqref{def:OLE} means the open-loop equilibrium happens to be expressed as a function of the states in the equilibrium, but not a closed-loop feedback equilibrium. To be more precise, plugging  \eqref{def:OLE} into \eqref{def:Xt} yields
\begin{equation}
\ud (\overline X_t^{\ast} - X_t^{i,\ast}) = -[a + q + (1-\oon)\eta_t] (\overline X_t^{\ast} - X_t^{i,\ast}) \ud t + \sigma \sqrt{1-\rho^2} \left(\oon \sum_{i=1}^N \ud W_t^i - \ud W_t^i\right).
\end{equation}
Thus, $\alpha_t^{i,\ast}$ is indeed $\mc{F}_t$-measurable. To avoid further confusion in the sequel, we denote by $\Xi_t^{i}$ the solution to the above SDE, then
\begin{equation}\label{eq:OLE}
\alpha_t^{i,\ast} = [q + (1-\oon)\eta_t] \Xi_t^i,
\end{equation}
and $\Xi_t^i$ is the unique strong solution to the SDE:
\begin{equation}\label{eq:Xi}
\ud \Xi_t^i = -\kappa_t\Xi_t^i \ud t + \sigma \sqrt{1-\rho^2} \left(\oon \sum_{i=1}^N \ud W_t^i - \ud W_t^i\right), \quad \Xi_0^i = \overline x_0 - x_0^i,
\end{equation}
with
\begin{equation}\label{def:kappa}
 \kappa_t = a + q + (1-\oon)\eta_t.
\end{equation}

Two properties regarding $\Xi_t^i$ will be used in sequel: firstly, $\sum_{i=1}^N \Xi_t^i = 0$, $\forall t \in [0,T]$. This is straightforward by  deriving the SDE for $\overline \Xi_t$ via summing  \eqref{eq:Xi} over $i \in \mc{I}$, and using $\overline \Xi_0 = 0$. Consequently, we also have $\sum_{i=1}^N \alpha_t^{i,\ast} = 0$, $\forall t \in [0,T]$. Secondly, one has that  $\displaystyle e^{\int_0^t \kappa_u \ud u} \Xi_t^{i}$ is a martingale, follows by the SDE \eqref{eq:Xi} and the boundedness of $\eta_t$ on $[0,T]$. 

We next verify that the limit $\alpha^{i,\infty}$ coincides with \eqref{eq:OLE} by showing the optimal control to the problem \eqref{def:J:lq} is $\alpha^{i,\ast}$ where other players'  are following $\alpha^{j,\ast}$, $j \neq i$, and by the uniqueness of limit under condition \eqref{def:cvgcondition}. Denote by $(\widetilde X_t^{i,\ast}, \psi_t^{i,\ast}, \phi_t^{i,\ast})$ be the solution to the FBSDEs \eqref{eq:FBSDE:decouple} with $\alpha^{j,n}$ replaced by $\alpha^{j,\ast}$, $j \in \mc{I} \setminus \{i\}$. Essentially, the problem is to show the player $i$'s optimal response, represented by the solution of FBSDEs, $(q + (1-\oon)K_t)\widetilde X_t^{i,\ast} - (1-\oon)\psi_t^{i,\ast}$ matches her Nash strategy $\alpha^{i,\ast}$. Note that this is not a fixed-point argument as usually seen in mean-field games, since only $\bm{\alpha}^{-i,\ast}$ is needed to solve $(\widetilde X_t^{i,\ast}, \psi_t^{i,\ast}, \phi_t^{i,\ast})$.

We first solve $\psi^{i,\ast}$ from the backward process in \eqref{eq:FBSDE:decouple}. The BSDE is of affine form, and thus possesses a unique solution:
\begin{align}
\psi_t^{i,\ast} & =  \EE\left[ \int_t^T - \frac{K_s}{N}\sum_{j \neq i}  \alpha_s^{j,\ast} e^{\int_s^t \gamma_u \ud u} \ud s \Bigg\vert \MCF_t\right] =  \EE\left[ \int_t^T  \frac{K_s}{N}  \alpha_s^{i,\ast} e^{\int_s^t \gamma_u \ud u} \ud s \Bigg\vert \MCF_t\right]\\
& =  \EE\left[ \int_t^T  \frac{K_s}{N}  [q + (1-\oon)\eta_s]\Xi_s^i e^{\int_s^t \gamma_u \ud u} \ud s \Bigg\vert \MCF_t\right] \\
& =   \int_t^T  \frac{K_s}{N}  [q + (1-\oon)\eta_s]\Xi_t^i e^{-\int_t^s \kappa_u + \gamma_u\ud u}  \ud s \\
& := F(t) \Xi_t^i.
\end{align}
The function $F(t)$ satisfies
\begin{equation}\label{def:F}
\dot F(t) = F(t)(\kappa_t + \gamma_t) - \frac{K_t}{N}(q + (1-\oon)\eta_t), \quad F(T) = 0,
\end{equation}
where $\gamma_t$ and $\kappa_t$ are given by \eqref{def:gamma} and \eqref{def:kappa} respectively, and $\eta_t$ solves \eqref{def:eta}. Note that \eqref{def:F} is a first order linear ordinary differential equation (ODE) with smooth coefficients, whose solution in uniqueness is ensured by standard ODE theory. A straightforward calculation shows $K_t - \eta_t$ solves \eqref{def:F}, thus $\psi_t^{i,\ast} = (K_t - \eta_t)\Xi_t^i$.

Now to solve the forward equation for $\widetilde X_t^{i,\ast}$, we first calculate
\begin{align}
\frac{\sum_{j \neq i}\alpha^{j,\ast}}{N} + (1-\oon)^2 \psi_t^{i,\ast} &= -\frac{\alpha^{i,\ast}}{N} + (1-\oon)^2(K_t - \eta_t) \Xi_t^i \\
& =  (-\frac{q}{N} + (1-\oon)^2K_t - (1-\oon)\eta_t) \Xi_t^i,
\end{align}
therefore 
\begin{align}
\ud \widetilde X_t^{i,\ast} &= [(-\frac{q}{N} + (1-\oon)^2K_t - (1-\oon)\eta_t) \Xi_t^i - \gamma_t  \widetilde X_t^{i,\ast}] \ud t \\
& \qquad  + \sigma\sqrt{1-\rho^2}\left(\oon\sum_{i=1}^N \ud W_t^i - \ud W_t^i\right).
\end{align}
Comparing it to \eqref{eq:Xi}, one deduces $ \widetilde X_t^{i,\ast} = \Xi_t^i$. Therefore, player $i$'s optimal response to her opponents' strategy $\bm{\alpha}^{-i,\ast}$ is
\begin{align}
(q + (1-\oon)K_t)\widetilde X_t^{i,\ast} - (1-\oon)\psi_t^{i,\ast} &= (q + (1-\oon)K_t)\Xi_t^i - (1-\oon)(K_t - \eta_t)\Xi_t^i \\
&= (q + (1-\oon)\eta_t) \Xi_t^i \equiv \alpha_t^{i,\ast},
\end{align}
which implies the limit of fictitious play gives an open-loop Nash equilibrium in the linear quadratic case.

\section{Numerical experiments}\label{sec:numerics}
In this section, we present the proof of methodology for deep fictitious play by applying our algorithm to the linear-quadratic game \eqref{def:Xt}-\eqref{def:J}, which was first introduced in \cite{CaFoSu:15} to study the systemic risk. We choose this model as our study for two reasons: firstly, convergence of fictitious play under this setting has been proved in Section~\ref{sec:LQ} under model assumptions. Secondly, closed-form solution exists for this model, which enables us to benchmark the performance of our proposed scheme. Numerical results are shown in three examples of $N=5, 10, 24$ players.

The Euler scheme (with time step $h = T/N_T$) of the dynamics \eqref{def:Xti:lq}-\eqref{def:Xtj:lq} follows from \eqref{eq:Xt:discrete} with:
\begin{align}
b^\ell(t,\bm{x}, \bm{\alpha}) = a(\overline{\bm{x}} - x^\ell) + \alpha^\ell, \quad \sigma^\ell(t,\bm{x}, \bm{\alpha})  = \sigma^0(t,\bm{x}, \bm{\alpha})  \equiv \sigma, \quad \ell \in \mc{I}.
\end{align}
The model parameters chosen by numerical experiments are
\begin{equation}
T = 1, \quad \sigma = 1, \quad  a = 1, \quad  q = 0, \quad \rho = 0,  \quad \eps = 1,\quad  c = 1.
\end{equation}
Remark that in the above choice, if one computes the factor in \eqref{def:cvgcondition}, which gives $\frac{1-e^{-2T\underline \gamma}}{\underline \gamma}C  = 0.9568, 1.5420$, $1.9995$ for $N = 5, 10 ,24$ respectively, then all the three cases in Proposition~\ref{prop:case} failed. However, we can still obtain convergent numerical results, which shows the robustness of the proposed algorithms and potential improvement of our theoretical analysis.
We choose $M = 2^{16}$ samples for training of the DNNs, and $M' = 10^6$ out-of-samples for final evaluation. A validation split ratio of $25\%$ and callbacks are set to avoid over-fitting. The subnetwork for policy approximation at each time step contains 2 hidden layers and $8 + 8$ neurons. During each stage, each network is trained for 200 epochs with a mini-batch of size 1024 . A total of 10 stages are played. The true (benchmark) optimal control is computed according \eqref{def:OLE}-\eqref{def:eta}, with $\eta_t$ given in the closed form.

{\bf Example 1 ($N=5$).} We set the initial states of the five players as $x_0 = (1,5,7,3,8)^\text{T}$ and discretize the time interval $[0,1]$ into $N_T = 50$ steps. In Figure~\ref{fig:N5cost}, we compare the cost functions computed by deep fictitious play to the closed-form solution. One can see that, the relative errors of cost function for all players drop quickly under 5\% after a few iterations, and then steadily under $2\%$ after only ten iterations. In Figure~\ref{fig:N5traj}, we show in the top-left panel 
optimal trajectories from total five players computed by deep fictitious play (black star lines) \emph{vs.} by closed-form formulae (colored solid lines) at one representative realization.  One can observe that players, although start away from each other, become closer as time evolves. This is consistent with the mechanism of costs functions, as they are in favor of being together. To quantitatively measure the performance of our algorithm, we show the mean and standard deviation of the difference between NN predictions and the true solutions in the rest panels based on a total of $10^6$ sample paths.  The means are almost zero, with slightly convex or concave curves depending on player's relative ranking initially. Players start below average tend to have convex feature. 

Standard numerical schemes can do well to approximate cost functions, but not on the derivatives, which are related to the controls, while our deep learning algorithm computes directly the control, which shows a good approximation. Figure~\ref{fig:N5control}  plots two visualized paths of controls for an illustration purpose.

\begin{figure}[H]
	\begin{tabular}{ccc}
		\includegraphics[width=0.28\textwidth]{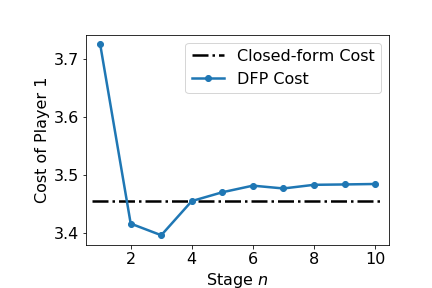} &
		\includegraphics[width=0.28\textwidth]{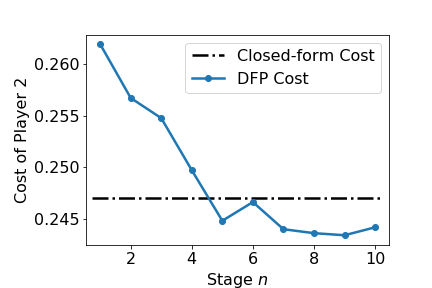}	&	
		\includegraphics[width=0.28\textwidth]{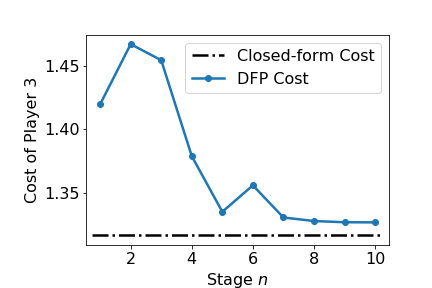} \\
		\includegraphics[width=0.28\textwidth]{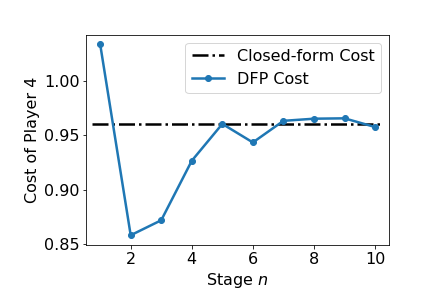} &
		\includegraphics[width=0.28\textwidth]{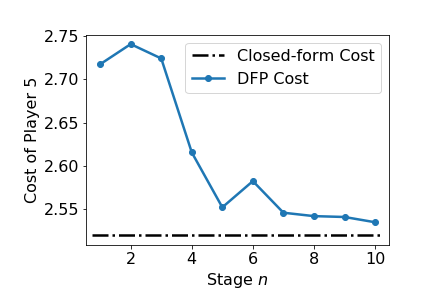} &
		\includegraphics[width=0.28\textwidth]{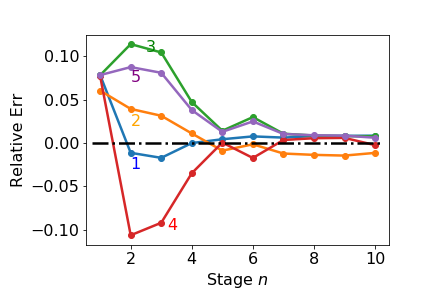} 
	\end{tabular}
\caption{Comparisons of cost functions for $N=5$ players in the linear quadratic game. The dotted dash lines are the analytical cost functions given by the closed-form solution for each individual player. The solid lines are the cost functions given by deep fictitious play for each player at the first $10$ iterations. The bottom-right panel shows the relative errors of cost function for the five players, which are pretty small at the $10^{\text{th}}$ iteration. }\label{fig:N5cost}
\end{figure}

\begin{figure}[H]
	\begin{tabular}{ccc}
		\includegraphics[width=0.28\textwidth]{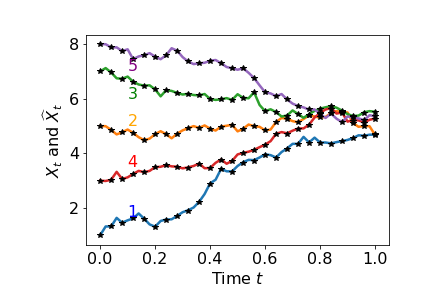} &
		\includegraphics[width=0.28\textwidth]{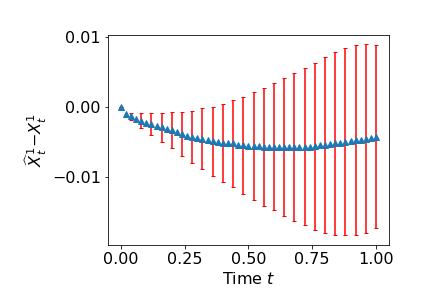}	&	
		\includegraphics[width=0.28\textwidth]{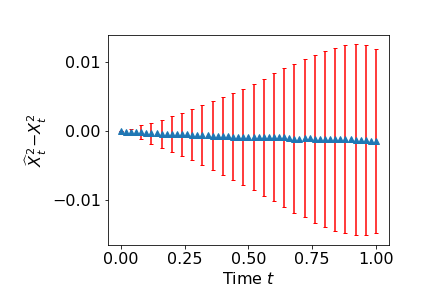} \\
		\includegraphics[width=0.28\textwidth]{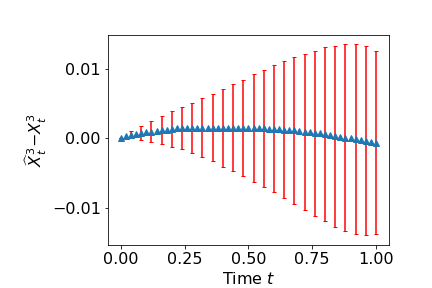} &
		\includegraphics[width=0.28\textwidth]{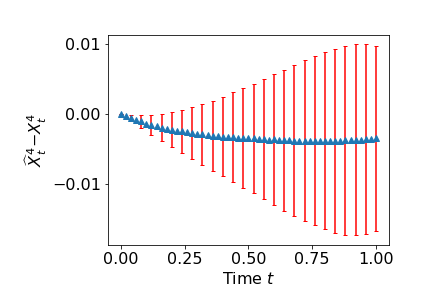} &
		\includegraphics[width=0.28\textwidth]{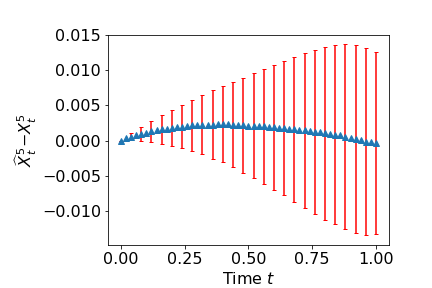} 
	\end{tabular}
	\caption{Comparisons of optimal trajectories for $N=5$ players in the linear quadratic game. Top-left panel: a single sample path of the true optimal trajectories $X_t$ (solid lines) \emph{vs.} the ones computed by deep fictitious play $\widehat X_t$ (star lines). The other panels show the mean (blue triangles) and standard deviation (red bars, plotted every other time step) of optimal trajectories errors for five players using a total sample of $10^6$ paths. Overall,  they show a good approximation of deep fictitious play to the linear quadratic game by $N=5$ players. }\label{fig:N5traj}
\end{figure}

\begin{figure}[H]
	\begin{tabular}{ccc}
		\includegraphics[width=0.3\textwidth]{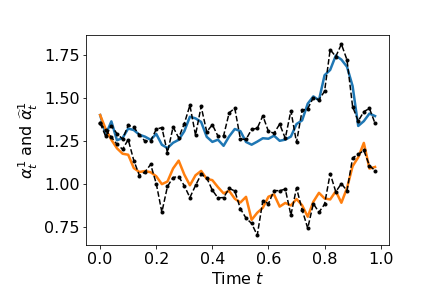} &
		\includegraphics[width=0.3\textwidth]{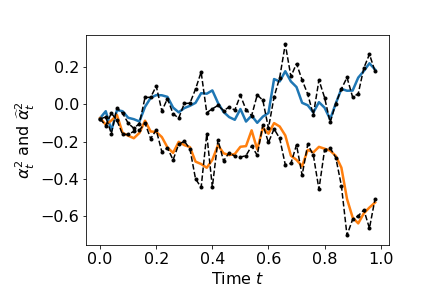}	&	
		\includegraphics[width=0.3\textwidth]{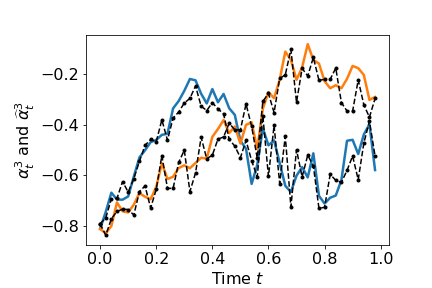} \\
	\end{tabular}
	\begin{tabular}{cc} \hspace{7em}
		\includegraphics[width=0.3\textwidth]{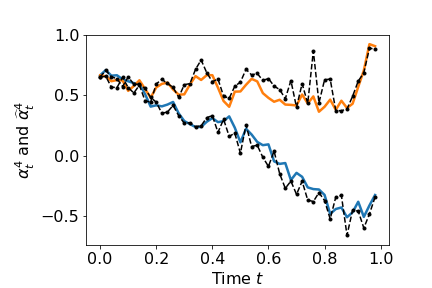}	&	
		\includegraphics[width=0.3\textwidth]{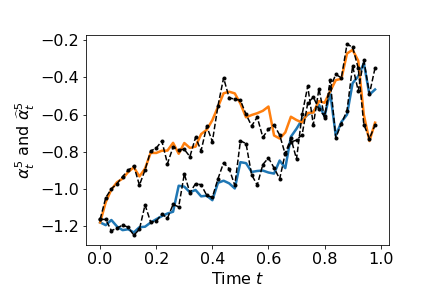}
	\end{tabular}
	\caption{Comparisons of optimal controls for $N=5$ players in the linear quadratic game. For a sake of clarity, we only show two sample paths of optimal controls for each player. The solid lines are optimal controls given by the closed-form solution, and the dotted dash lines are computed by deep fictitious play.}\label{fig:N5control}
\end{figure}

{\bf Example 2 ($N=10$).} The initial state for $i^{th}$ player is $x_0^i = 0.5 + 0.05(i-1)$. We use $N_T = 20$ time steps for the discretization of the time interval $[0,1]$. Such choices enable us to investigate the sensitivity of deep learning algorithm on initial positions and time step. In Figure~\ref{fig:N10_cost_traj}, we compare the cost functions computed by deep fictitious play to the closed-form solution, where, after only ten iterations, the maximum relative error of cost function for all players have been reduced to less than $3\%$, and the computed optimal trajectories (one visualized sample path) of selected four players by fictitious play coincide with those of the closed-form solution. The standard deviation of difference between approximated and true optimal trajectories as less then $2\%$ for $t\in[0, 1]$ for all players, and we present a selection of six in Figure~\ref{fig:N10traj}. 

Note that, although the time step $h$ is twice larger than $N=5$, the relative error does not increase significantly. However, we do not observe that the trajectories are getting closer and closer as in the case of $N = 5$, since they already start in the neighborhood of each other. We do not observe the curve neither, which justify our assertion that the curvature depends on $\overline x_0 - x_0^i$. We also show two visualized sample paths of optimal control in Figure~\ref{fig:N10control}, which presents a good approximation of the policy.

\begin{figure}[H]
	\begin{tabular}{cc}
		\includegraphics[width=0.45\textwidth]{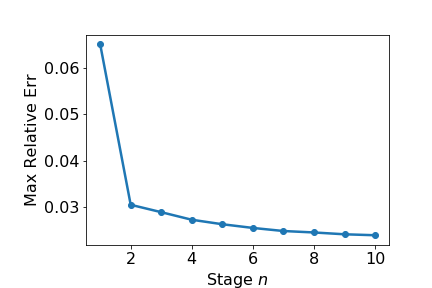}	&	
		\includegraphics[width=0.45\textwidth]{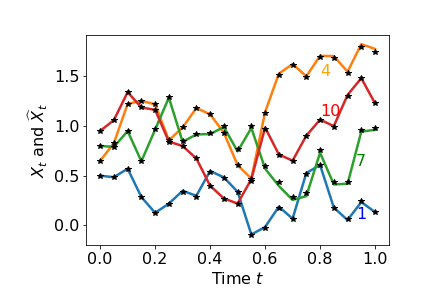}
	\end{tabular}
	\caption{Comparisons of cost functions and optimal trajectories for $N=10$ players in the linear quadratic game. Left: the maximum relative errors of the cost functions for ten players; Right: for a sake of clarity, we only present the comparison of optimal trajectories for the $1^\text{st}$, $4^\text{th}$, $7^\text{th}$ and $10^\text{th}$ players, where the solid lines are given by the closed-form solution and the stars are computed by deep fictitious play. }\label{fig:N10_cost_traj}
\end{figure}

\begin{figure}[H]
	\begin{tabular}{ccc}
		\includegraphics[width=0.3\textwidth]{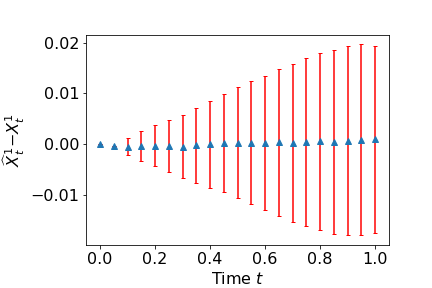} &
		\includegraphics[width=0.3\textwidth]{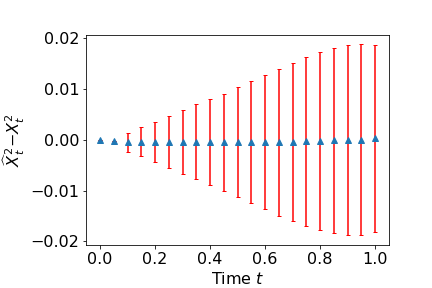}	&	
		\includegraphics[width=0.3\textwidth]{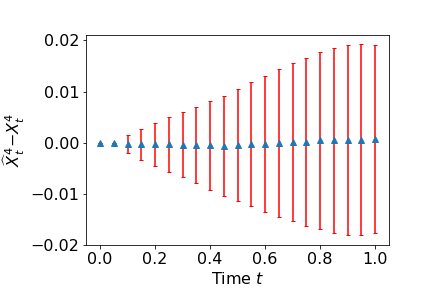} \\
		\includegraphics[width=0.3\textwidth]{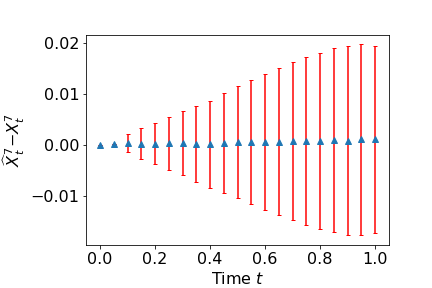} &
		\includegraphics[width=0.3\textwidth]{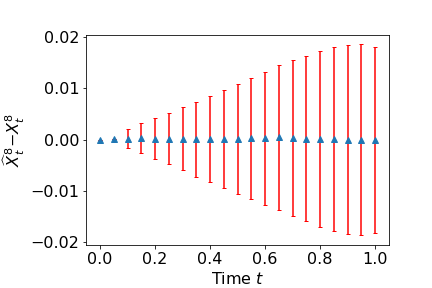} &
		\includegraphics[width=0.3\textwidth]{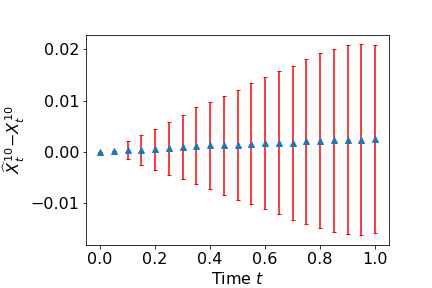} 
	\end{tabular}
	\caption{Comparisons of optimal trajectories for $N=10$ players in the linear quadratic game. For a sake of clarity, we only show the mean (blue triangles) and standard deviation (red bars) of  optimal trajectories errors for the $1^\text{st}$, $2^\text{nd}$, $4^\text{th}$, $7^\text{th}$, $8^\text{th}$ and $10^\text{th}$ player, respectively. The results are based on a total sample of $65536$ paths, and show that deep fictitious play provides a uniformly good accuracy of optimal trajectories.}\label{fig:N10traj}
\end{figure}

\begin{figure}[h t b]
	\begin{tabular}{ccc}
		\includegraphics[width=0.3\textwidth]{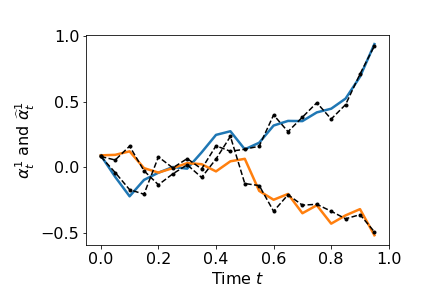} &
		\includegraphics[width=0.3\textwidth]{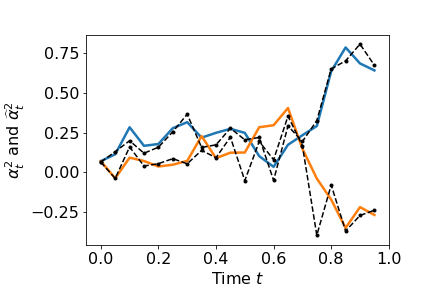}	&	
		\includegraphics[width=0.3\textwidth]{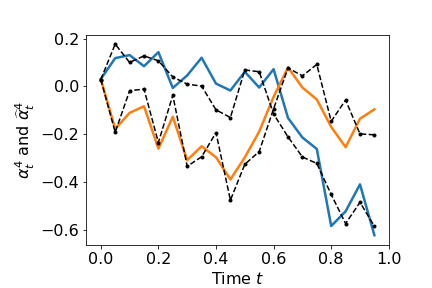} \\
		\includegraphics[width=0.3\textwidth]{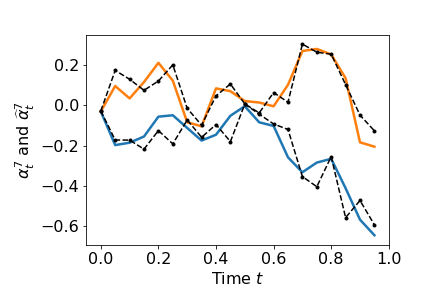} &
		\includegraphics[width=0.3\textwidth]{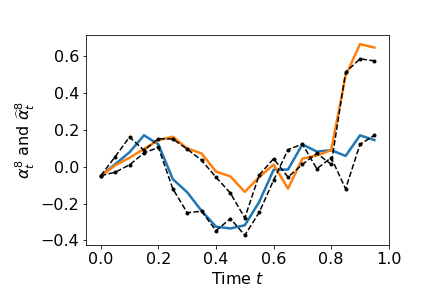} &
		\includegraphics[width=0.3\textwidth]{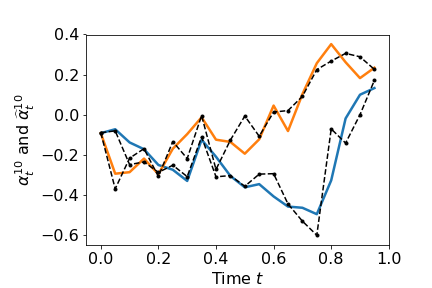} 
	\end{tabular}
	\caption{Comparisons of optimal controls for $N=10$ players in the linear quadratic game. For a sake of clarity, we only show two sample paths of optimal controls for the $1^\text{st}$, $2^\text{nd}$, $4^\text{th}$, $7^\text{th}$, $8^\text{th}$ and $10^\text{th}$ player, respectively. The solid lines are optimal controls given by the closed-form solution, and the dotted dash lines are computed by deep fictitious play. }\label{fig:N10control}
\end{figure}

{\bf Example 3 ($N=24$).} The initial positions for the $i^{th}$ player is $x_0^i = 0.5i$. We set the time steps $N_T = 20$, after observing the relative errors did not increase too much from $N_T = 50$ to $N_T = 20$. The problem by nature is high-dimensional: the $k^\text{th}$ ``Sequential'' subnetwork maps $\RR^{Nk}$ to $\RR$. To accelerate the computation, we distribute the training to 8 GPUs. Similar studies to the $N=10$ case are presented in Figures~\ref{fig:N24:cost:traj}-\ref{fig:N24control}. Some key features that have been observed from previous numerical experiments: the maximum of relative error drops below $3\%$ after ten iterations; the average error of estimated trajectories are convex/concave functions of time $t$; the standard deviation of estimated error aggregates from steps to steps. In fact, the convexity/concavity with respect to time $t$ is caused by two factors: the propagation of errors, which produces an magnitude increase in error mean; and the existence of terminal cost, which puts more weights on $X_T$ than $X_t, t \in (0,T)$, resulting in a better estimate of $X_T$ and a decreasing effect.

\begin{figure}[H]
	\begin{tabular}{cc}
		\includegraphics[width=0.45\textwidth]{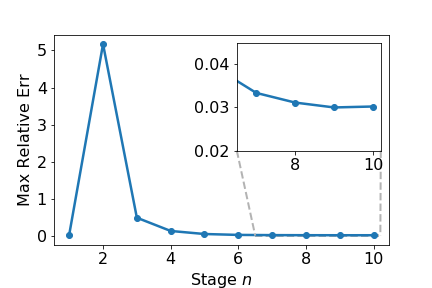}	&	
		\includegraphics[width=0.45\textwidth]{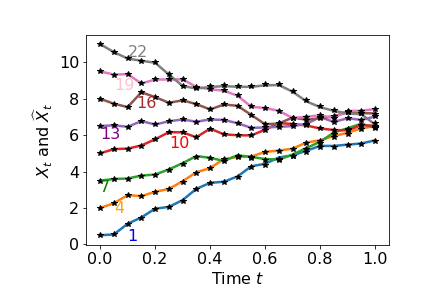}
	\end{tabular}
	\caption{Comparisons of cost functions and optimal trajectories for $N=24$ players in the linear quadratic game. Left: the maximum relative errors of the cost functions for ten players; Right: for a sake of clarity, we only present the comparison of optimal trajectories for the $1^\text{st}$, $4^\text{th}$, $7^\text{th}$, $10^\text{th}$, $13^\text{th}$, $16^\text{th}$, $19^\text{th}$ and $22^\text{th}$ players, where the solid lines are given by the closed-form solution and the stars are computed by deep fictitious play. }\label{fig:N24:cost:traj}
\end{figure}

\begin{figure}[h t b]
	\begin{tabular}{ccc}
		\includegraphics[width=0.3\textwidth]{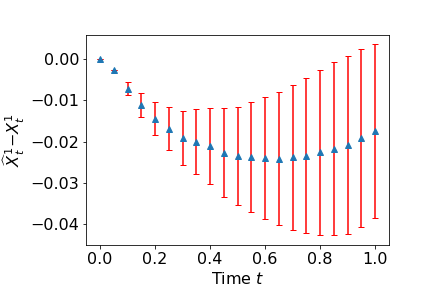} &
		\includegraphics[width=0.3\textwidth]{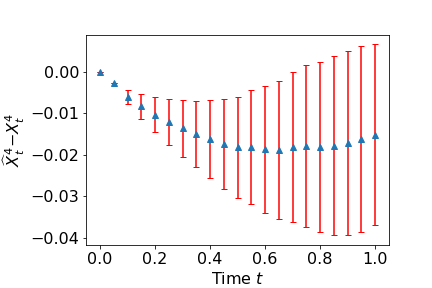}	&	
		\includegraphics[width=0.3\textwidth]{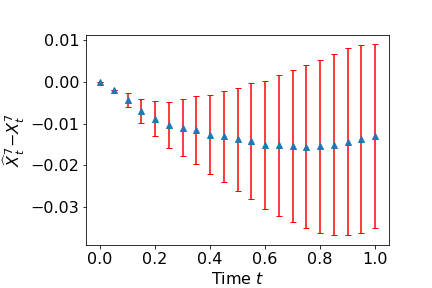} \\
		\includegraphics[width=0.3\textwidth]{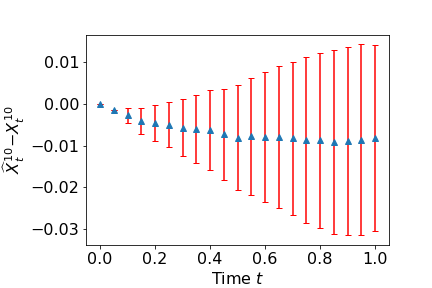} &
		\includegraphics[width=0.3\textwidth]{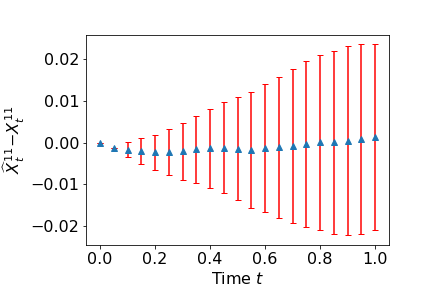} &
		\includegraphics[width=0.3\textwidth]{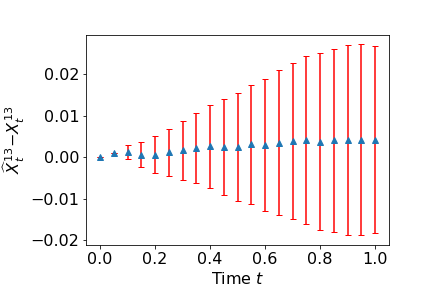} \\
		\includegraphics[width=0.3\textwidth]{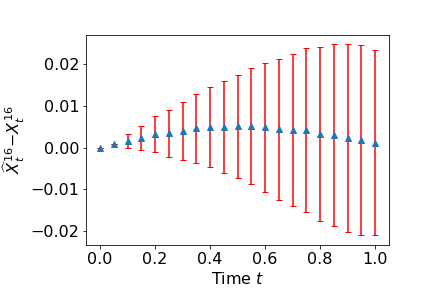} &
		\includegraphics[width=0.3\textwidth]{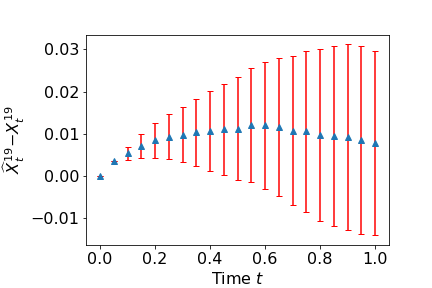} &
		\includegraphics[width=0.3\textwidth]{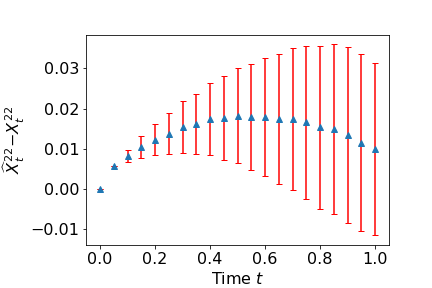} \\
	\end{tabular}
	\caption{Comparisons of optimal trajectories for $N=24$ players in the linear quadratic game. For a sake of clarity, we only show the mean (blue triangles) and standard deviation (red bars) of optimal trajectories errors for the $1^\text{st}$, $4^\text{th}$, $7^\text{th}$, $10^\text{th}$, $11^\text{th}$, $13^\text{th}$, $16^\text{th}$, $19^\text{th}$ and $22^\text{th}$ player, respectively. The results are based on a total sample of $65536$ paths,  show that deep fictitious play provides a uniformly good accuracy of optimal trajectories.}\label{fig:N24traj}
\end{figure}

\begin{figure}[h t b]
	\begin{tabular}{ccc}
		\includegraphics[width=0.3\textwidth]{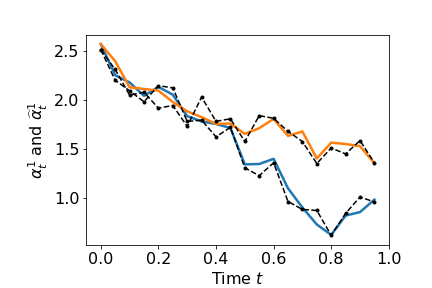} &
		\includegraphics[width=0.3\textwidth]{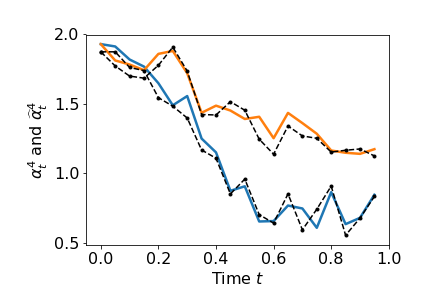}	&	
		\includegraphics[width=0.3\textwidth]{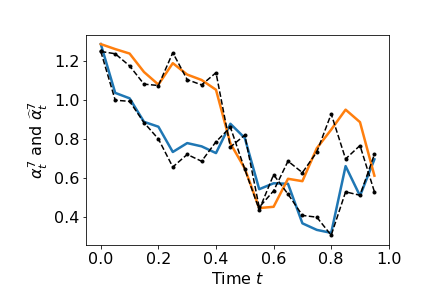} \\
		\includegraphics[width=0.3\textwidth]{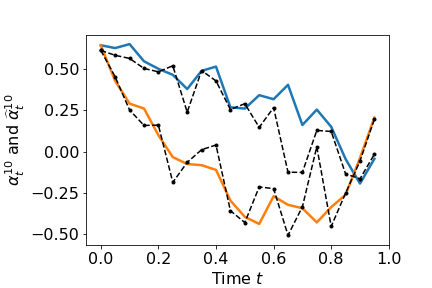} &
		\includegraphics[width=0.3\textwidth]{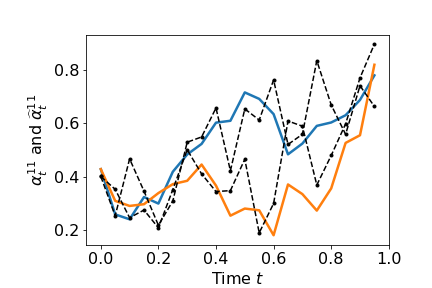} &
		\includegraphics[width=0.3\textwidth]{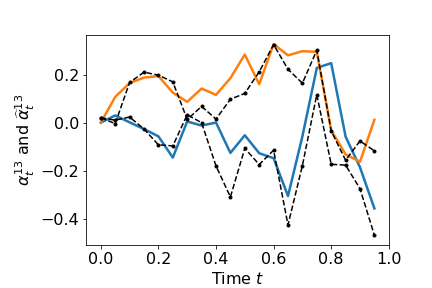} \\
		\includegraphics[width=0.3\textwidth]{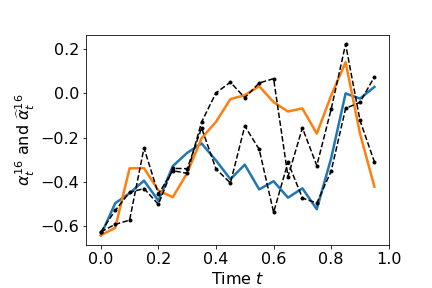} &
		\includegraphics[width=0.3\textwidth]{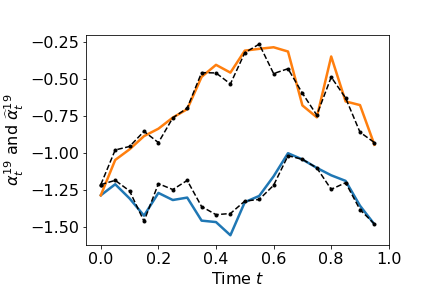} &
		\includegraphics[width=0.3\textwidth]{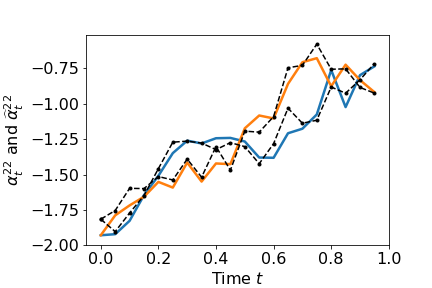} 
	\end{tabular}
	\caption{Comparisons of optimal controls for $N=24$ players in the linear quadratic game. For a sake of clarity, we only show two sample paths of optimal controls for the $1^\text{st}$, $4^\text{th}$, $7^\text{th}$, $10^\text{th}$, $11^\text{th}$, $13^\text{th}$, $16^\text{th}$, $19^\text{th}$ and $22^\text{th}$ player, respectively. The solid lines are optimal controls given by the closed-form solution, and the dotted dash lines are computed by deep fictitious play. }\label{fig:N24control}
\end{figure}

To better illustrate that our algorithm can overcome the curse of dimensionality, we compare the performance across different $N$. Particularly, we compute $$\max_{i \in \mc{I}}\max_{k \leq N_T}\abs{X_{kh}^i - \widehat X_{kh}^i}$$
where $X$ denotes the state process following the open-loop Nash equilibrium, while $\widehat X$ is the deep fictitious play counterpart. 
The $L^1$ error is $1.09\times 10^{-2}$ for $N = 5$, $1.49\times 10^{-2}$ for $N = 10$ and $2.08\times 10^{-2}$ for $N = 24$. Table~\ref{tab:parameter} gives the running time and other hyper-parameters used in the numerical examples.

\begin{table}[h]
	\caption{Hyperparameters and runtime for the numerical examples presented in Section 4.}\label{tab:parameter}
	\begin{center}
		~\newline 
		\begin{tabular}{@{}c|ccc@{}}
			\toprule
			Problem & N = 5    & N = 10 & N = 24 \\ \midrule
			$N_T$       &  50 &  20 &  20   \\
		Max Relative Err & 1.15\% &2.45\% & 2.95\%\\
			\# of GPUs used & 1 & 1 & 8 \\
			runtime (hours) $^\dagger$ & 2.15 & 14.03 & 12.10 \\
			$L^1$ error of $\widehat X$ & 1.09e-2 & 1.49e-2 & 2.08e-2 \\
			\bottomrule
		\end{tabular}
	\end{center}
	\small{
		$^\dagger$ The numerical experiments were conducted using Amazon EC2 services with P2 instances. We remark that the runtime is subject to further reduction with a multi-GPU system or more efficient GPUs.
	}
\end{table}

 \section{Conclusion, discussion and extension}\label{sec:rmk}
 
 In this paper, the deep fictitious play theory is proposed to compute the Nash equilibrium of asymmetric $N$-player non-zero-sum stochastic differential games. We apply the strategy of fictitious play by letting individual player optimize her own payoff while fixing the control of the other players at each stage, and then repeat the game until their responses do not change too much from stage to stage. Finding the best response for each player at each stage is a stochastic optimal control problem, which we approximate by deep neural networks (DNNs). By the nature of open-loop strategies, the problem is recasted into repeated training of $N$ decoupled neural networks (NNs),  where inputs of each NN depend on the other NNs' outputs from previous training. Using Keras and parallel GPU simulation, the deep learning algorithm can be applied to any $N$-player stochastic differential game with different symmetries and heterogeneities. The numerical accuracy and efficiency is illustrated by comparing to the closed-form solution of the linear quadratic case. We also prove the convergence of fictitious play under appropriate assumptions, and show that the convergent limit forms an open-loop Nash equilibrium. We remark that the implementation of this algorithm causes no extra difficulties beyond the linear-quadratic game, but the verification of convergence to the true equilibrium is in general hard due to the lack of benchmark solution. Although one may observe the convergence of the proposed algorithm by tracking the relative change of cost (cf. Step~\ref{algo:step} in Algorithm~\ref{def:algorithm}), it may actually be trapped in a local (but not true) equilibrium. 
 
 In the following, we shall discuss the extensions to other neural network architectures, other strategies of fictitious play and closed-loop Nash equilibrium.

\subsection{Other neural network architectures}
	In the open-loop framework, the searching space for optimal policies contains all $\MCF_t$-progressively measurable processes, which possesses a path-dependent feature. When using a feedforward architecture, in order to better capture this feature, one needs to partition $[0,T]$ into a sufficiently large number of $N_T$ intervals. Then, a sub-network is used to approximate the optimal policy at each time point \eqref{def:strategy:approx}, whose size becomes larger as the time approaches the terminal time $T$ since more history needs to be fed as input. Therefore, the training time increases significantly when one uses large $N_T$. To improve the performance, architectures based on recurrent neural networks can be considered in solving the stochastic control problem \eqref{def:sc1}--\eqref{def:sc2}, for example, using long short-term memory (LSTM),  gated recurrent units (GRUs), {\it etc.} This will be part of our future work \cite{HaHu:20}.

\subsection{Belief based on time average of past play} \label{sec:avg}
In the formulation \eqref{def:SFP}, players' belief is based on their actions during last round, i.e. at stage $n+1$, players myopically respond to their opponents' policies at stage $n$ without considering all decisions before $n$. This is in fact a bit discrepant from Brown's definition \cite{Br:49,Br:51}, where players responses take into account all past policies. Denote by $\bm{\widetilde \alpha}^{-i,n}$ is the weighted average of past play,
\begin{align}\label{def:SFP:avg}
\bm{\widetilde \alpha}^{-i,n} = \frac{1}{n}\sum_{k=1}^n \bm{\alpha}^{-i,k},
\end{align}
then Brown's original idea corresponds to the control problem:
\begin{align}
\alpha^{i,n+1} := \argmin_{\beta^i \in \mbA} J^i (\beta^i; \widetilde{\bm{\alpha}}^{-i,n}), \quad \forall i \in \mc{I}, n \in \NN.
\end{align}
where $J^i$ is defined as in \eqref{def:J:SFP}.

In general, convergence in the strategy $\bm{\alpha}^n$ implies convergence in the average of past play $\widetilde {\bm{\alpha}}^n$, but not vice versa. Therefore, convergence in $\widetilde{\bm{\alpha}}^n$ does not necessarily lead to a Nash equilibrium. Our numerical tests show that, if the algorithm converges in $\bm\alpha^n$, then using $\widetilde {\bm{\alpha}}^n$ tends to give a better rate for linear quadratic cases. In practice, within the framework of deep fictitious play, one can generalize \eqref{def:SFP:avg} to any weighted average of past policies: $\sum_{k=0}^n c_k\bm{\alpha}^{-i,k}$, where $(c_k)_{k=0}^n$ is a $n$-simplex with $c_n >0$. We plan to further investigate the comparison between different beliefs for practical problems in future.

\subsection{ Belief updated alternatively} 
We shall also mention that, there are actually two versions of fictitious play, the alternating fictitious play (AFP), originally invented in \cite{Br:49}, and the simultaneous fictitious play (SFP) mentioned as a minor variant of AFP in \cite{Br:49}. In contrast to \eqref{def:SFP}, the players under AFP update their beliefs alternatively. For example, in the case $N=2$, the learning process is:
\begin{align}
\alpha^{1,n+1} &:= \argmin_{\beta^1 \in \mbA} J^1(\beta^1; \alpha^{2,n}),  \quad \alpha^{2,n} := \argmin_{\beta^2 \in \mbA} J^2(\beta^2; \alpha^{1,n}),  \quad n \geq 1,
\end{align}
and the computation follows $\alpha^{2,0}$(initial belief) $\rightarrow \alpha^{1,1}\rightarrow \alpha^{2,1}\rightarrow \alpha^{1,2}\rightarrow \alpha^{2,2}\rightarrow \ldots$. The dependence of $\alpha^{2,n}$ on $\alpha^{1,n}$ makes one not able to update them simultaneously, which is the main difference from SFP. 

Indeed, SFP can be considered as a simpler learning process than AFP, as players are treated symmetrically in time. This usually enhances analytical convenience as well as numerical efficiency (with possible parallel implementation in Step 5-9 of Algorithm~\ref{def:algorithm}). Gradually, the original AFP seems to disappear from the literature, and people focus on SFP, even though SFP may generate subtle problems which do not arise under AFP. For a comparison study, we refer to \cite{Be:07}, where they also related this subtly to Monderer and Sela's {Improvement Principle} \cite{MoSe:97}. We focused on SFP in this paper, where the beliefs can be updated in parallel, and leave the AFP learning process for future studies.

\subsection{ The algorithm for closed-loop Nash equilibrium}\label{sec:close:loop} Depending on the space we search for $\beta^i$ in \eqref{def:SFP}, the algorithm can lead to a Nash equilibrium in different setting. Indeed, if consider $[0,T] \times (\RR^d)^N \ni  (t, \bm{x}) \to \beta^{i} \in A \subset \RR^k$ as a function of current states, then the limit yields a feedback strategy for Nash equilibrium. Mathematically, 
\begin{equation}\label{def:SFP:cl}
\alpha^{i,n+1}(t,\bm{x}) := \argmin_{\beta^i(t,\bm{x}) \in A} J^i(\beta^i(X_t^{i, \beta^i}, \bm{X}_t^{-i, \bm{\alpha}^{-i,n}}); \bm{\alpha}^{-i,n}(X_t^{i, \beta^i}, \bm{X}_t^{-i, \bm{\alpha}^{-i,n}})),
\end{equation}
where $\bm{X}_t^{-i, \bm{\alpha}^{-i,n}}$ represents players $j \neq i$ state processes following policies $\bm{\alpha}^{-i,n}$. 

This setup can be analyzed by the the partial differential equation (PDE) approach. Assuming enough regularity, the minimal cost can be reformulated as the classical solution to HJB equation where others' strategies are given by deterministic functions obtained from previous round. Consequently, at each stage, the task is to solve $N$ independent HJB equations, which can still be implemented in parallel. Moreover, if the players are statistically identical, one actually only needs to solve one PDE. Denote by $V^{i,n+1}(t,\bm{x})$ the value function of problem \eqref{def:SFP:cl} at time $t$ with initial states $\bm{X}_t = \bm{x}$, by dynamic programming, it satisfies
\begin{align}
\partial_t V^{i,n+1} &+ \inf_\beta\Bigg\{b^i(t,\bm{x}, \beta)\partial_{x^i}V^{i,n+1} + f^i(t,\bm{x}, \beta) + \half \text{Tr}\left[ \partial^2_{x^i,x^i}V^{i,n+1}\sigma^i(t,\bm{x},\beta)\sigma^i(t,\bm{x},\beta)^\dagger\right] \\
&+ \sum_{\substack{j = 1\\ j \neq i}}^N \text{Tr}\left[\partial^2_{x^i,x^j}V^{i,n+1} \sigma^i(t,\bm{x}, \beta) \Sigma^{i,j} \sigma^j(t,\bm{x}, \alpha^{j,n})^\dagger\right]\Bigg\}
+ \sum_{\substack{j=1 \\ j\neq i}}^N b^j(t,\bm{x}, \alpha^{j,n})\partial_{x^j}V^{i,n+1}  \\
&+ \half \sum_{\substack{j,k = 1\\ j\neq i\\k \neq i}}^N \text{Tr}\left[\partial^2_{x^j,x^k}V^{i,n+1} \sigma^j(t,\bm{x}, \alpha^{j,n}) \Sigma^{j,k}\sigma^k(t,\bm{x}, \alpha^{k,n})^\dagger\right]= 0, \\
\alpha^{i,n} \equiv \alpha&^{i,n}(t,\bm{x}) := \argmin_{\beta \in A}\left\{b^i(t,\bm{x}, \beta)\partial_{x^i}V^{i,n} + f^i(t,\bm{x}, \beta)\right\}, \quad \Sigma^{j,k} \ud t := \ud \average{W^j, W^k}_t.
\end{align}
Then, numerically, one can design traditional finite difference/element methods, or use deep learning which has been shown excellent performance in overcoming the curse of dimensionality in high-dimensional PDEs \cite{EHaJe:17,HaJeE:18}. After all, the optimal response function $\alpha^{i,n+1}$ is given in terms of $\partial_{x^i}V^{i,n+1}, \partial^2_{x^i, x^j}V^{i,n+1}$. However, a common drawback of working on the value function $J^i$ is that numerical schemes usually well approximate the solution but not the derivative of the solution, which is more sensitive. 

An alternative way is to work directly on the control.  By a stochastic maximum principle argument, the optimal control is linked to the solution (not the derivative) of FBSDEs, see, {\it e.g.}, \cite[Section~2.2]{CaDe1:17}. Then it is promising to apply the recent deep learning algorithm for the coupled FBSDEs \cite{HaLo:18}. In this case, at each stage, the task is to solve $N$ independent FBSDEs and parallel implementation is still possible.

Both approaches rely on the property of the reformulated problem: the solution's regularity in the PDE approach and the Hamiltonian's convexity in the FBSDEs approach. A third possibility is to work with the optimization \eqref{def:SFP:cl} directly as we do in the open-loop case. That is, using the deep NN to approximate the control and find the optimal parameters that minimize \eqref{def:SFP:cl}. However, due to the feedback reaction, the Algorithm~\ref{def:algorithm} and  architectures proposed in Section~\ref{sec:algorithm} are no longer suitable. It is this ``indirect'' reaction nature of the open-loop strategy that enables us to design $N$ separate NNs and a scalable algorithm. While working with feedback controls,  the realized opponents' strategies $\bm{\alpha}^{-i,n}(t, \bm{X}_t)$ depend on $\beta^i$. Further explained by Figure~\ref{fig:algorithm}, this means that,  $\bm\alpha_1^{-i}$, previously considered as intermediate outputs from NNs of other players at previous training, now depends on $\beta_0^i$ through $X_1^{i}$. Consequently, to take into account the direct reaction of her opponents, one needs to feed $\beta_0^i$ to player $j^{th}$ NN, $j \neq i$ for intermediate output $\bm\alpha_1^{-i}$. This makes the $N$-neural networks coupled with each other, and hard to implement in parallel.  

Apparently, using deep fictitious play for Markovian Nash equilibrium is not a simple modification of Algorithm~\ref{def:algorithm}, and two of the three aforementioned approaches (PDE and direct) are studied in the follow-up works \cite{HaHu:19,HaHuLo:20}.

\section*{Acknowledgment}
I am grateful to Professor Marcel Nutz for the stimulating and fruitful discussions on fictitious play and convergence of linear quadratic case.

\bibliographystyle{plain}
\bibliography{Reference}

\begin{thebibliography}{10}

\bibitem{deep2:2018}
A.~Bachouch, C.~Hur{\'e}, N.~Langren{\'e}, and H.~Pham.
\newblock Deep neural networks algorithms for stochastic control problems on
  finite horizon: numerical applications.
\newblock {\em arXiv:1812.05916}, 2018.

\bibitem{Be:09}
Y.~Bengio.
\newblock Learning deep architectures for {AI}.
\newblock {\em Foundations and trends{\textregistered} in Machine Learning},
  2(1):1--127, 2009.

\bibitem{Be:05}
U.~Berger.
\newblock Fictitious play in 2{$\times$}n games.
\newblock {\em Journal of Economic Theory}, 120(2):139--154, 2005.

\bibitem{Be:07}
U.~Berger.
\newblock Brown's original fictitious play.
\newblock {\em Journal of Economic Theory}, 135(1):572--578, 2007.

\bibitem{BrCa:18}
A.~Briani and P.~Cardaliaguet.
\newblock Stable solutions in potential mean field game systems.
\newblock {\em Nonlinear Differential Equations and Applications NoDEA},
  25(1):1, 2018.

\bibitem{Br:49}
G.~W. Brown.
\newblock Some notes on computation of games solutions.
\newblock Technical report, Rand Corp Santa Monica, CA, 1949.

\bibitem{Br:51}
G.~W. Brown.
\newblock Iterative solution of games by fictitious play.
\newblock {\em Activity Analysis of Production and Allocation}, 13(1):374--376,
  1951.

\bibitem{CaHa:17}
P.~Cardaliaguet and S.~Hadikhanloo.
\newblock Learning in mean field games: the fictitious play.
\newblock {\em ESAIM: Control, Optimisation and Calculus of Variations},
  23(2):569--591, 2017.

\bibitem{CaDe:13}
R.~Carmona and F.~Delarue.
\newblock Probabilistic analysis of mean-field games.
\newblock {\em SIAM Journal on Control and Optimization}, 51(4):2705--2734,
  2013.

\bibitem{CaDe1:17}
R.~Carmona and F.~Delarue.
\newblock {\em Probabilistic Theory of Mean Field Games with Applications {I}}.
\newblock Springer, 2017.

\bibitem{CaDe2:17}
R.~Carmona and F.~Delarue.
\newblock {\em Probabilistic Theory of Mean Field Games with Applications
  {II}}.
\newblock Springer, 2017.

\bibitem{CaFoSu:15}
Ren{\'e} Carmona, Jean-Pierre Fouque, and Li-Hsien Sun.
\newblock Mean field games and systemic risk.
\newblock {\em Communications in Mathematical Sciences}, 13(4):911--933, 2015.

\bibitem{Keras}
F.~Chollet et~al.
\newblock Keras.
\newblock \url{https://keras.io}, 2015.

\bibitem{CrAn:03}
R.~Cressman and C.~Ansell.
\newblock {\em Evolutionary dynamics and extensive form games}, volume~5.
\newblock MIT Press, 2003.

\bibitem{Cy:89}
G.~Cybenko.
\newblock Approximations by superpositions of a sigmoidal function.
\newblock {\em Mathematics of Control, Signals and Systems}, 2:183--192, 1989.

\bibitem{DaGoPa:2009}
C.~Daskalakis, P.~W. Goldberg, and C.~H. Papadimitriou.
\newblock The complexity of computing a {Nash} equilibrium.
\newblock {\em SIAM Journal on Computing}, 39:195--259, 2009.

\bibitem{Dozat16}
T.~Dozat.
\newblock Incorporating {Nesterov} momentum into {Adam}.
\newblock In {\em International Conference on Learning Representations 2016 -
  Workshop Track}, 2016.

\bibitem{EHaJe:17}
W.~E, J.~Han, and A.~Jentzen.
\newblock Deep learning-based numerical methods for high-dimensional parabolic
  partial differential equations and backward stochastic differential
  equations.
\newblock {\em Communications in Mathematics and Statistics}, 5(4):349--380,
  2017.

\bibitem{FoYo:98}
D.~P. Foster and H.~P. Young.
\newblock On the nonconvergence of fictitious play in coordination games.
\newblock {\em Games and Economic Behavior}, 25(1):79--96, 1998.

\bibitem{HaE:16}
J.~Han and W.~E.
\newblock Deep learning approximation for stochastic control problems.
\newblock {\em Deep Reinforcement Learning Workshop, NIPS}, 2016.

\bibitem{HaHu:20}
J.~Han and R.~Hu.
\newblock Deep learning-based methods for stochastic control problems with
  delay.
\newblock In preparation, 2020.

\bibitem{HaHu:19}
J.~Han and R.~Hu.
\newblock Deep fictitious play for finding markovian {Nash} equilibrium in
  multi-agent games.
\newblock {\em Mathematical and Scientific Machine Learning Conference
  (MSML2020)}, to appear. {\it arXiv:1912.01809}.

\bibitem{HaHuLo:20}
J.~Han, R.~Hu, and J.~Long.
\newblock Convergence of deep fictitious play for stochastic differential
  games.
\newblock {\em arXiv preprint arXiv:2008.05519}, 2020.

\bibitem{HaJeE:18}
J.~Han, A.~Jentzen, and W.~E.
\newblock Solving high-dimensional partial differential equations using deep
  learning.
\newblock {\em Proceedings of the National Academy of Sciences},
  115(34):8505--8510, 2018.

\bibitem{HaLo:18}
J.~Han and J.~Long.
\newblock Convergence of the deep {BSDE} method for coupled {FBSDEs}.
\newblock {\em Probability, Uncertainty and Quantitative Risk}, 5(1):1--33,
  2020.

\bibitem{HeSi:16}
J.~Heinrich and D.~Silver.
\newblock Deep reinforcement learning from self-play in imperfect-information
  games.
\newblock {\em arXiv:1603.01121}, 2016.

\bibitem{HoSa:02}
J.~Hofbauer and W.~H. Sandholm.
\newblock On the global convergence of stochastic fictitious play.
\newblock {\em Econometrica}, 70(6):2265--2294, 2002.

\bibitem{HoMoSe:98}
S.~Hon-Snir, D.~Monderer, and A.~Sela.
\newblock A learning approach to auctions.
\newblock {\em Journal of Economic Theory}, 82(1):65--88, 1998.

\bibitem{Ho:91}
K.~Hornik.
\newblock Approximation capabilities of multilayer feedforward networks.
\newblock {\em Neural networks}, 4(2):251--257, 1991.

\bibitem{Hu:19}
R.~Hu.
\newblock Deep learning for ranking response surfaces with applications to
  optimal stopping problems.
\newblock {\em Quantitative Finance}, to appear.

\bibitem{HuCaMa:07}
M.~Huang, P.~E. Caines, and R.~P. Malham{\'e}.
\newblock Large-population cost-coupled {LQG} problems with nonuniform agents:
  individual-mass behavior and decentralized $\epsilon$-{Nash} equilibria.
\newblock {\em IEEE Transactions on Automatic Control}, 52(9):1560--1571, 2007.

\bibitem{HuMaCa:06}
M.~Huang, R.~P. Malham{\'e}, and P.~E. Caines.
\newblock Large population stochastic dynamic games: closed-loop
  {McKean-Vlasov} systems and the {Nash} certainty equivalence principle.
\newblock {\em Communications in Information and Systems}, 6(3):221--252, 2006.

\bibitem{deep:2018}
C.~Hur{\'e}, H.~Pham, A.~Bachouch, and N.~Langren{\'e}.
\newblock Deep neural networks algorithms for stochastic control problems on
  finite horizon, part {I}: convergence analysis.
\newblock {\em arXiv:1812.04300}, 2018.

\bibitem{batch}
S.~Ioffe and C.~Szegedy.
\newblock Batch normalization: Accelerating deep network training by reducing
  internal covariate shift.
\newblock In {\em International Conference on Machine Learning}, pages
  448--456, 2015.

\bibitem{Jo:93}
J.~S. Jordan.
\newblock Three problems in learning mixed-strategy {Nash} equilibria.
\newblock {\em Games and Economic Behavior}, 5(3):368--386, 1993.

\bibitem{Adam}
D.~Kingma and J.~Ba.
\newblock {Adam}: A method for stochastic optimization.
\newblock {\em International Conference of Learning Representations (ICLR)},
  2015.

\bibitem{Ko:91}
A.N. Kolmogorov.
\newblock On the representation of continuous functions of several variables as
  superpositions of continuous functions of one variable and addition.
\newblock In {\em Mathematics and Its Applications (Soviet Series)}, pages
  383--387. Springer, 1991.

\bibitem{KrSj:98}
V.~Krishna and T.~Sj{\"o}str{\"o}m.
\newblock On the convergence of fictitious play.
\newblock {\em Mathematics of Operations Research}, 23(2):479--511, 1998.

\bibitem{KrSuHi:12}
A.~Krizhevsky, I.~Sutskever, and G.~E. Hinton.
\newblock Imagenet classification with deep convolutional neural networks.
\newblock In {\em Advances in Neural Information Processing Systems}, pages
  1097--1105, 2012.

\bibitem{LaZa:17}
M.~Lanctot, V.~Zambaldi, A.~Gruslys, A.~Lazaridou, K.~Tuyls, J.~P{\'e}rolat,
  D.~Silver, and T.~Graepel.
\newblock A unified game-theoretic approach to multiagent reinforcement
  learning.
\newblock In {\em Advances in Neural Information Processing Systems}, pages
  4190--4203, 2017.

\bibitem{LaLi1:2006}
J.-M. Lasry and P.-L. Lions.
\newblock Jeux \`{a} champ moyen. {I.} {L}e cas stationnaire.
\newblock {\em C. R. Math. Acad. Sci. Paris}, 9:619--625, 2006.

\bibitem{LaLi2:2006}
J.-M. Lasry and P.-L. Lions.
\newblock Jeux \`{a} champ moyen. {II.} {H}orizon fini et contr\^{o}le optimal.
\newblock {\em C. R. Math. Acad. Sci. Paris}, 10:679--684, 2006.

\bibitem{LaLi:2007}
J.-M. Lasry and P.-L. Lions.
\newblock Mean field games.
\newblock {\em Japanese Journal of Mathematics}, 2:229--260, 2007.

\bibitem{LeBeHi:15}
Y.~LeCun, Y.~Bengio, and G.~Hinton.
\newblock Deep learning.
\newblock {\em Nature}, 521(7553):436, 2015.

\bibitem{MaMoYo:99}
J.~Ma, J.-M. Morel, and J.~Yong.
\newblock {\em Forward-backward stochastic differential equations and their
  applications}.
\newblock Number 1702. Springer Science \& Business Media, 1999.

\bibitem{MaWuZhZh:15}
J.~Ma, Z.~Wu, D.~Zhang, and J.~Zhang.
\newblock On well-posedness of forward-backward {SDEs}--{A} unified approach.
\newblock {\em The Annals of Applied Probability}, 25(4):2168--2214, 2015.

\bibitem{MgJeCo:18}
D.~Mguni, J.~Jennings, and E.~M. de~Cote.
\newblock Decentralised learning in systems with many, many strategic agents.
\newblock In {\em Thirty-Second AAAI Conference on Artificial Intelligence},
  2018.

\bibitem{MiRo:91}
P.~Milgrom and J.~Roberts.
\newblock Adaptive and sophisticated learning in normal form games.
\newblock {\em Games and Economic Behavior}, 3(1):82--100, 1991.

\bibitem{Mi:61}
K.~Miyasawa.
\newblock On the convergence of the learning process in a 2{$\times$}2
  non-zero-sum two-person game.
\newblock Technical report, Princeton University, NJ, 1961.

\bibitem{MoSe:96}
D.~Monderer and A.~Sela.
\newblock A 2{$\times$}2 game without the fictitious play property.
\newblock {\em Games and Economic Behavior}, 14(1):144--148, 1996.

\bibitem{MoSe:97}
D.~Monderer and A.~Sela.
\newblock Fictitious play and no-cycling conditions.
\newblock Technical report, Faculty of Industrial Engineering \& Management,
  Technion - Israel Institute of Technology. URL http://iew3.technion.ac.il,
  1997.

\bibitem{MoSh:96}
D.~Monderer and L.~S. Shapley.
\newblock Fictitious play property for games with identical interests.
\newblock {\em Journal of Economic Theory}, 68(1):258--265, 1996.

\bibitem{MoSh2:96}
D.~Monderer and L.~S. Shapley.
\newblock Potential games.
\newblock {\em Games and Economic Behavior}, 14(1):124--143, 1996.

\bibitem{Ni:online}
M.~Nielsen.
\newblock Neural networks and deep learning.
\newblock \url{http://neuralnetworksanddeeplearning.com/}.

\bibitem{NuMaTa:18}
M.~Nutz, J.~San~Martin, and X.~Tan.
\newblock Convergence to the mean field game limit: {A} case study.
\newblock {\em Annals of Applied Probability}, 30(1):259--286, 2020.

\bibitem{PaPe:90}
E.~Pardoux and S.~Peng.
\newblock Adapted solution of a backward stochastic differential equation.
\newblock {\em Systems \& Control Letters}, 14(1):55--61, 1990.

\bibitem{PaTa:99}
E.~Pardoux and S.~Tang.
\newblock Forward-backward stochastic differential equations and quasilinear
  parabolic {PDEs}.
\newblock {\em Probability Theory and Related Fields}, 114(2):123--150, 1999.

\bibitem{PeWu:99}
S.~Peng and Z.~Wu.
\newblock Fully coupled forward-backward stochastic differential equations and
  applications to optimal control.
\newblock {\em SIAM Journal on Control and Optimization}, 37(3):825--843, 1999.

\bibitem{Pinkus:99}
A.~Pinkus.
\newblock Approximation theory of the {MLP} model in neural networks.
\newblock {\em Acta Numerica}, 8:143--195, 1999.

\bibitem{Po:07}
W.~B. Powell.
\newblock {\em Approximate Dynamic Programming: Solving the curses of
  dimensionality}, volume 703.
\newblock John Wiley \& Sons, 2007.

\bibitem{Adamcvg}
S.~J. Reddi, S.~Kale, and S.~Kumar.
\newblock On the convergence of {Adam} and beyond.
\newblock In {\em 6th International Conference on Learning Representations
  (ICLR)}, 2018.

\bibitem{Ro:51}
J.~Robinson.
\newblock An iterative method of solving a game.
\newblock {\em Annals of Mathematics}, pages 296--301, 1951.

\bibitem{Sh:64}
Lloyd Shapley.
\newblock Some topics in two-person games.
\newblock {\em Advances in Game Theory}, 52:1--29, 1964.

\bibitem{Zh:17}
J.~Zhang.
\newblock {\em Backward Stochastic Differential Equations: From Linear to Fully
  Nonlinear Theory}.
\newblock Springer, 2017.

\end{thebibliography}

\end{document}